\documentclass[11pt]{article}

\usepackage{amsmath}
\usepackage{amsfonts}
\usepackage{amssymb}
\usepackage{amscd}
\usepackage{latexsym}
\usepackage{mathrsfs}
\usepackage{amsthm}

\newcommand\ie{{\em i.e.}}

\def\A{{A}}

\def\N{\mathbb{N}}
\def\T{\mathbb{T}}

\def\Z{\mathbb{Z}}
\def\C{\mathbb{C}}
\def\D{{\mathcal D}}
\def\F{\mathscr F}
\def\f{u}
\def\H{\mathcal H}
\def\HH{\mathscr H}
\def\K{\mathcal K}
\def\e{{\mathrm e}}

\def\R{\mathbb{R}}
\def\V{\mathcal V}

\def\S{\mathcal S}
\def\J{\mathcal J}
\def\LL{\mathfrak L}
\def\U{\mathscr U}

\def\KK{\mathfrak K}

\def\OO{\mathcal O}
\def\UU{\mathcal U}
\def\MM{\mathcal M}
\def\NN{\mathcal N}
\def\SS{\mathcal S}
\def\TT{\mathcal{T}}

\def\SSS{\mathscr S}

\def\NNN{\mathfrak N}

\def\la{\langle}
\def\ra{\rangle}

\def\Id{\mathbb I}
\def\I{\mathscr I}

\def\B{\mathcal B}

\def\({\left(}
\def\[{\left[}
\def\){\right)}
\def\]{\right]}
\def\l|{\left\lvert}
\def\r|{\right\rvert}
\def\lp{\left\lVert}
\def\rp{\right\rVert}

\def\<{\langle}
\def\>{\rangle}

\def\Op{\mathfrak{Op}}

\def\XX{\mathfrak{X}}
\def\xx{\mathfrak{x}}

\def\ee{\mathfrak{e}}

\def\J{\mathscr{J}}

\def\d{\mathrm d}
\def\r{{\mathrm r}}
\def\bl{\boldsymbol{\lambda}}
\def\WW{\mathcal W}

\DeclareMathOperator{\mul}{mul}

\DeclareMathOperator{\rank}{rank}

\DeclareMathOperator{\Aut}{Aut}

\DeclareMathOperator{\sgn}{sgn}
\DeclareMathOperator{\Ran}{Ran}

\newtheorem{Theorem}{Theorem}[section]
\newtheorem{Remark}[Theorem]{Remark}
\newtheorem{Lemma}[Theorem]{Lemma}

\newtheorem{Proposition}[Theorem]{Proposition}
\newtheorem{Definition}[Theorem]{Definition}

\begin{document}

\title{Spectral and scattering theory for Schr\"odinger operators on perturbed topological crystals}

\author{D. Parra$^1$, S. Richard$^2$\footnote{Supported by JSPS Grant-in-Aid for Young Scientists A
no 26707005, and on leave of absence from
Univ.~Lyon, Universit\'e Claude Bernard Lyon 1, CNRS UMR 5208, Institut Camille Jordan,
43 blvd.~du 11 novembre 1918, F-69622 Villeurbanne cedex, France.}}

\date{\small}
\maketitle \vspace{-1cm}

\begin{quote}
\emph{
\begin{itemize}
\item[$^1$]
Univ.~Lyon, Universit\'e Claude Bernard Lyon 1,
CNRS UMR 5208, Institut Camille Jordan,
43 blvd. du 11 novembre 1918, F-69622 Villeurbanne cedex, France
\item[$^2$] Graduate school of mathematics, Nagoya University,
Chikusa-ku,  Nagoya 464-8602, Japan
\item[] \emph{E-mail:} parra@math.univ-lyon1.fr, richard@math.nagoya-u.ac.jp
\end{itemize}
}
\end{quote}

\begin{abstract}
In this paper we investigate the spectral and the scattering theory of Schr\"odinger operators acting on perturbed
periodic discrete graphs. The perturbations considered are of two types: either a multiplication operator by a short-range or a long-range function,
or a short-range type modification of the measure defined on the vertices and on the edges of the graph.
Mourre theory is used for describing the nature of the spectrum of the underlying operators.
For short-range perturbations, existence and completeness of local wave operators are also proved.
\end{abstract}

\smallskip


\section{Introduction}\label{sec_intro}
\setcounter{equation}{0}

The aim of this paper is to describe the spectral theory of a Schr\"odinger operator $H$ acting on a perturbed periodic discrete graph.
The main strategy is to exploit the fibered decomposition of the periodic underlying operator $H_0$ in the unperturbed graph to get a Mourre estimate.
Then, by applying perturbative techniques, the description of the nature of the spectrum of $H$ can be deduced:
it consists of absolutely continuous spectrum, of a finite number (possibly zero) of eigenvalues of infinite multiplicity,
and of eigenvalues of finite multiplicity which can accumulate only at a finite set of thresholds.
The scattering theory for the pair $(H, H_0)$ is also investigated.

The study of Laplace operators on infinite graphs has recently attracted lots of attention.
Let us mention for example the problem of essential self-adjointness for very general infinite graphs \cite{Gol, Keller},
or the more precise study of the spectrum for bounded Laplacians \cite{Ando123,MRT07}.
For periodic graphs it is well-known that this spectrum has a band structure
with at most a finite number of eigenvalues of infinite multiplicity \cite{HN09}.
This structure is preserved if one considers periodic Schr\"odinger operators \cite{KS14,KS15,KS15x}.
Our interest is in what happens when such periodic Schr\"odinger operators are perturbed.

The perturbations we consider are of two types. On the one hand we add a potential that decays at infinity either as a short-range or as a long-range function.
To the best of our knowledge this has not been studied for general periodic graphs and only the case of $\Z^d$ has been fully investigated in $\cite{BS99}$.
In that respect, our main theorem generalizes such results to arbitrary periodic graphs.
Note that some related results on the inverse scattering problem are available for $\Z^d$ in \cite{IK12} and the hexagon lattice in \cite{An13},
but only compactly supported perturbations are considered.

The second types of perturbations we consider correspond to the modification of the graph itself.
This kind of perturbations has recently been studied in \cite{SS15x} for investigating the stability of the essential spectrum.
In \cite{AIM15} results similar to ours are exhibited, but the perturbations considered there are only compactly supported
and some implicit conditions on the Floquet-Bloch variety are assumed. These two restrictions do not appear in our work.
Let us still mention the related work \cite{CT13} where compactly supported perturbation are considered in the framework of a regular tree.

As pointed before, the two main tools that we use is the Floquet-Bloch decomposition of periodic Schr\"odinger operator and Mourre theory.
This decomposition is an important tool for the analysis of the periodic graphs and we mention only a few articles that use it \cite{An13,AIM15,HN09,KS14,KSS98}.
For Mourre theory, we refer to \cite{ABG96} for the general theory and to \cite{GN98} for this theory applied to analytically fibered operators
from which our work is inspired. In the discrete setting, this theory has already been used for example in \cite{AF00,MRT07}.
In the special case of the graph $\Z^d$, it plays a central role in \cite{BS99}.
Mourre theory for more general periodic graphs has also been mentioned in \cite{HN09} for proving that the Laplace operator has a
purely absolutely continuous spectrum outside some discrete spectrum. However, since no perturbation were considered in that paper,
the theory was not further developed. Our paper can thus also be seen as an extension of that work.

Finally, we would like to stress that several definitions of periodic graphs can be found in the literature.
We have opted for the setting of topological crystals which has the advantage that no embedding in the Euclidean space is needed.
We refer to \cite{Su12} for a thorough introduction to topological crystals and to many examples of such structures.

Let us still mention that in this paper we restrict
our attention to Laplace operators acting on the vertices of the graph. In the companion paper \cite{P16} still in preparation,
Gauss-Bonnet operators are studied, as well as the Laplacian acting on edges \cite{BGJ15x}.
Note that the Gauss-Bonnet operator is a Dirac-type operator that acts both on vertices and edges.
This operator has recently been investigated in \cite{AT15,GH14}.

We finally describe the content of this paper. In Section \ref{sec_main_result} we describe the framework of our investigations and provide our main result.
Under suitable assumptions it consists in the description of the spectral type of the operators under investigation, and in the existence
of suitable wave operators. More precise information on the purely periodic setting are then presented in Section \ref{sec_H_0} and we show that
the periodic operator $H_0$ can be decomposed into a family of magnetic Schr\"odinger operators.
In Section \ref{section:analyticdecomposition} it is proved that the latter operator is unitarily equivalent to an analytically fibered operator.
In order to be self-contained, a brief review on real analyticity and a few definitions
and results are provided.
Section \ref{sec_Mourre} is dedicated to the conjugate operator theory, also called Mourre theory.
For completeness, we first describe the abstract framework of this theory,
and provide then a thorough construction of the necessary conjugate operator.
In fact, this construction is inspired from \cite{GN98} but part of the argumentation has been simplified for our context.
In addition, we can take advantage of the recent reference \cite{RT09} which supplies a lot of information on toroidal pseudodifferential operators.
Based on all these preliminary constructions, the proof of the main theorem is given in Section \ref{section:proofbis}.
A first preliminary subsection discuss the regularity of some abstract operators with respect to the newly constructed conjugate operator,
and these results are finally applied to operators appearing in our context of the perturbation of a periodic graph.

\section{Framework and main result}\label{sec_main_result}
\setcounter{equation}{0}

In this section we describe the framework of our investigations and state our main result.

A graph $X=\big(V(X),E(X)\big)$ is composed of a set of vertices $V(X)$ and a set of unoriented edges $E(X)$.
Graphs with loops and parallel edges are accepted.
Generically we shall use the notation $x,y$ for elements of $V(X)$, and $\e=\{x,y\}$ for elements of $E(x)$.
If both $V(X)$ and $E(X)$ are finite sets, the graph $X$ is said to be finite.

A morphism $\omega: X\to \XX$ between two graphs $X$ and $\XX$ is composed of two maps $\omega:V(X)\to V(\XX)$
and $\omega:E(X)\to E(\XX)$ such that it preserves the adjacency relations between vertices and edges,
namely $\omega(\e)=\{\omega(x),\omega(y)\}$.
Let us stress that we use the same notation for the two maps $\omega:V(X)\to V(\XX)$
and $\omega:E(X)\to E(\XX)$, and that this should not lead to any confusion.
An isomorphism is a morphism that is a bijection on the vertices and on the edges.
The group of isomorphisms of a graph $X$ into itself is denoted by $\Aut(X)$.
For a vertex $x\in V(X)$ we also set $E(X)_x:=\{\e\in E(X)\mid x\in\e\}$.
If $E(X)_x$ is finite for every $x\in V(X)$ we say that $X$ is locally finite.

A morphism $\omega:X\to\XX$ between two graphs is said to be \emph{a covering map} if
\begin{enumerate}
\item[(i)] $\omega:V(X)\to V(\XX)$ is surjective,
\item[(ii)] for all $x\in V(X)$, the restriction $\omega|_{E(X)_x}:E(X)_x\to E(\XX)_{\omega(x)}$ is a bijection.
\end{enumerate}
In that case we say that $X$ is a \emph{covering graph} over the \emph{base graph} $\XX$.
For such a covering, we define the \emph{transformation group} of the covering as the subgroup of $\Aut(X)$,
denoted by $\Gamma$, such that for every $\mu\in\Gamma$ the equality $\omega\circ\mu=\omega$ holds.
We now define a topological crystal, and refer to \cite[Sec. 6.2]{Su12} for more details.

\begin{Definition}\label{topocrystal}
A $d$-dimensional topological crystal is a quadruplet $(X,\XX,\omega,\Gamma)$ such that:
\begin{enumerate}
\item[(i)] $X$ is an infinite graph,
\item[(ii)] $\XX$ is a finite graph,
\item[(iii)] $\omega:X\to \XX$ is a covering map,
\item[(iv)] The transformation group $\Gamma$ of $\omega$ is isomorphic to $\Z^d$,
\item[(v)] $\omega$ is regular, \ie~for every $x$, $y\in V(X)$ satisfying $\omega(x)=\omega(y)$ there exists $\mu\in\Gamma$ such that $x=\mu y$.
\end{enumerate}
\end{Definition}

We will usually say that $X$ is a topological crystal if it admits a $d$-dimensional topological crystal structure $(X,\XX,\omega,\Gamma)$.
Note that all topological crystal are locally finite, with an upper bound for the number of elements in $E(X)_x$ independent of $x$.
Indeed, the local finiteness and the fixed upper bound follow from the definition of a covering and the assumption (ii) of the previous definition.

From the set of unoriented edges $E(X)$ of an arbitrary graph $X$ we construct the set of oriented edges $\A(X)$
by considering for every unoriented edge  $\{x,y\}$ both $(x,y)$ and $(y,x)$ in $\A(X)$.
The elements of $\A(X)$ are still denoted by $\e$.
The origin vertex of such an oriented edge $\e$ is denoted by $o(\e)$, the terminal one by $t(\e)$,
and $\overline{\e}$ denotes the edge obtained from $\e$ by interchanging the vertices, \ie~$o(\overline{\e})=t(\e)$ and $t(\overline{\e})=o(\e)$.
For $x\in V(X)$ we set $A(X)_x\equiv \A_x:=\{\e\in\A(X)\mid o(\e)=x\}$.
Clearly, any morphism $\omega$ between a graph $X$ and a graph $\XX$, and in particular any covering map,
can be extended to a map sending oriented edges of $\A(X)$
to oriented edges of $\A(\XX)$. For this extension we keep the convenient notation $\omega:\A(X)\to \A(\XX)$.

A \emph{measure} $m$ on a graph $X$ is a strictly positive function defined on vertices and on unoriented edges.
On oriented edges, the measure satisfies $m(\e)=m(\overline{\e})$.
From now on, let us assume that the graph $X$ is locally finite. For such a graph
the Laplace operator is defined on the space of $0$-\emph{cochains} $C^0(X):=\{f\mid V(X)\to\C\}$
by
\begin{equation*}
\[\Delta(X,m)f\](x)=\sum_{\e\in\A_x}\frac{m(\e)}{m(x)}\big(f\big(t(\e)\big)-f(x)\big), \qquad  \forall f\in C^0(X).
\end{equation*}
Furthermore, when
\begin{equation}\label{eq_def_deg}
\deg_{m}: V(X)\to\R_+ , \quad \deg_{m} (x):=\sum_{\e\in\A_{x}}\frac{m(e)}{m(x)}
\end{equation}
is bounded, then the operator $\Delta(X,m)$ is a bounded self-adjoint operator in the Hilbert space
\begin{equation*}
l^2(X,m)=\Big\{f \in C^0(X) \mid \lp f\rp^2:=\sum_{x\in V(X)}\;\!m(x)|f(x)|^2<\infty\Big\}
\end{equation*}
endowed with the scalar product
\begin{equation*}
\langle f,g\rangle = \sum_{x\in V(X)}\;\!m(x)f(x)\;\!\overline{g(x)}\qquad \forall f,g\in l^2(X,m).
\end{equation*}

Let us now consider a topological crystal $X$, a $\Gamma$-periodic measure $m_0$ and a $\Gamma$-periodic function $R_0: V(X)\to \R$.
The periodicity means that for every $\mu\in\Gamma$, $x\in V(X)$ and $\e \in E(X)$ we have $m_0(\mu x)=m_0(x)$, $m_0(\mu \e) = m_0(\e)$ and $R_0(\mu x)=R_0(x)$.
We can then provide the definition of a \emph{periodic Schr\"odinger operator}. It consists in the operator
\begin{equation}\label{h0vertices}
H_0:=-\Delta (X,m_0)+R_0.
\end{equation}
Note that we use the same notation for the function $R_0$ and for the corresponding multiplication operator.
As a consequence of our assumptions, the expression $H_0$ defines a bounded self-adjoint operator in the Hilbert space $l^2(X,m_0)$.

Our aim is to study rather general perturbations of the operator $H_0$.
In fact, we shall consider two types of perturbations. The first one consists in replacing the multiplication operator by a function
$R$ which converges rapidly enough to $R_0$ at infinity. The precise formulation
will be provided in the subsequent statement.
The other type of perturbation is more substantial and consists in modifying the measure on the graph.
For that purpose, we shall consider a second strictly positive measure $m$ on $X$,
and which converges in a suitable sense to the $\Gamma$-periodic measure $m_0$.
The corresponding perturbed operator acts then in the Hilbert space $l^2(X,m)$ and has the form
\begin{equation}\label{full}
H=-\Delta (X,m)+ R.
\end{equation}

Let us stress that this modification of the measure naturally leads to a two-Hilbert space problem since the measures $m_0$ and $m$
enter into the definition of the underlying Hilbert spaces.
Fortunately, since the graph structure is not modified, a unitary transformation between both spaces is at hand.
Namely, we consider $\J:l^{2}(X,m)\to l^{2}(X,m_0)$ defined by
\begin{equation}\label{eq_def_J}
[\J f](x)=\Big(\frac{m(x)}{m_0(x)}\Big)^{\frac12}f(x), \qquad f\in l^2(X,m).
\end{equation}
Note that this map is well-defined and unitary since $m_0(x)$ and $m(x)$ are assumed to be strictly positive for any $x\in V(X)$.
The inverse of $\J$ is given by  $[\J^*f](x) = \big(\frac{m_0(x)}{m(x)}\big)^{\frac12}f(x)$.
The fact that $\J$ is unitary plays an essential role in the comparison of both operators.

We have now almost all the ingredients for stating our main result.
The missing ingredient is the definition of the entire part of a vertex and of an edge, denoted respectively by
$\[x\]\in\Gamma$ and $\[\e\]\in\Gamma$, see \eqref{entirex} and \eqref{entiree} for the details.
Indeed, in order to properly introduce these notions some additional definitions are necessary and we have decided
to postpone them to the next section.
We still mention that the isomorphism between $\Gamma$ and $\Z^d$ allows us to borrow the Euclidean norm $|\cdot|$ of $\Z^d$ and to
endow $\Gamma$ with it. As a consequence of this construction, the notations $|[x]|$ and $|[e]|$ are well-defined,
and the notion of rate of convergence towards infinity is available.

\begin{Theorem}\label{principalvertices}
Let $X$ be a topological crystal. Let $H_0$ and $H$ be defined by \eqref{h0vertices} and \eqref{full} respectively. Assume that $m$ satisfies
\begin{equation}\label{measureshort}
\int_{1}^{\infty}\d\lambda \sup_{\lambda<|\[\e\]|<2\lambda}\left|\frac{m(\e)}{m(o(\e))}-\frac{m_0(\e)}{m_0(o(\e))}\right|<\infty\ .
\end{equation}
Assume also that the difference $R-R_0$ is equal to $R_s+R_l$ which satisfy
\begin{equation} \label{Rshipotesis}
\int_{1}^{\infty}\d\lambda \sup_{\lambda<|\[x\]|<2\lambda}\left|R_s(x)\right|<\infty,
\end{equation}
and
\begin{equation}\label{Rlhipotesis}
R_l(x)\xrightarrow{x\to\infty}0, \quad \hbox{ and } \quad
\int_{1}^{\infty}\d\lambda \sup_{\lambda<|\[\e\]|<2\lambda}\big|R_l\big(t(\e)\big)-R_l\big(o(\e)\big)\big|<\infty\ .
\end{equation}
Then, there exists a discrete set $\tau\subset\R$ such that for every closed interval $I\subset\R\backslash\tau$ the following assertions hold:
\begin{enumerate}
\item $H_0$ has not eigenvalues in $I$ and $H$ has at most a finite number of eigenvalues in $I$
and each of these eigenvalues is of finite multiplicity,
\item $\sigma_{sc}(H_0)\cap I = \sigma_{sc}(H)\cap I=\emptyset$,
\item If $R_l\equiv0$, the local wave operators
$$
W_{\pm}(H,H_0;\J^*,I)=s-\lim_{t\to\pm\infty}e^{iH t}\J^* e^{-iH_0 t}E_{H_0}(I)
$$
exist and are asymptotically complete, \ie~$\Ran (W_-)=\Ran( W_+)=E_{H}^{ac}(I)l^2(X,m)$.
\end{enumerate}
\end{Theorem}

Note that the operator $\J^*$ enters into the definition of the wave operators (instead of the more traditional notation $\J$)
since we have defined $\J$ from $l^2(X,m)$ to $l^2(X,m_0)$. This choice is slightly more natural in our context.

The hypothesis \eqref{measureshort} and \eqref{Rshipotesis} are usually referred to as a \emph{short-range} type of decay.
In particular it is satisfied for functions that decay faster than $C(1+|[x]|)^{-1-\epsilon}$ for some constant $C$ independent of $x$.
It is worth mentioning that that condition \eqref{measureshort} is quite general and is automatically satisfied if the difference $m-m_0$ itself satisfies a
short-range type of decay. For example if we assume that $|m(\e)-m_0(\e)|\leq C (1+|[e]|)^{-1-\epsilon}$ and
$|m(x)-m_0(x)|\leq C' (1+|[x]|)^{-1-\epsilon}$, then \eqref{measureshort} is satisfied.
On the other hand \eqref{Rlhipotesis} is usually called a \emph{long-range} decay since the difference $R_l\big(t(\e)\big)-R_l\big(o(\e)\big)$
should be thought as the derivative of $R_l$ at the point $o(\e)$ in the direction $\e$.
To sum up we can say that we cover perturbations by short-range and long-range potentials but only by short-range perturbation of the metric.

\begin{Remark}
A more drastic modification would be to allow $m(x)=0$ for some $x\in V(X)$,
and this would roughly correspond to the suppression of some vertices in the graph. Reciprocally,
it would also be natural to consider a perturbation of the operator $H_0$ on the topological crystal $X$
by the addition of some vertices to $X$. Note that these modifications are more difficult to encode since
there would be no natural unitary operator available between the corresponding Hilbert spaces.
These perturbations will not be considered in the present paper but we intend to come back to them
in the future.
\end{Remark}

\section{Periodic operator and its direct integral decomposition}\label{sec_H_0}
\setcounter{equation}{0}

The aim of this section is to provide some additional information on the periodic Schr\"odinger operator and to show that this operator can be decomposed
into the direct integral of magnetic Schr\"odinger operators defined on the small graph $\XX$.
This decomposition is an important tool for studying its spectral properties, as shown for example in \cite{An13,HN09,KS14,KSS98}.

Let us consider a topological crystal $(X,\XX,\omega,\Gamma)$.
The notation $x$, resp.~$\xx$, will be used for the elements of $V(X)$, resp.~of $V(\XX)$,
and accordingly the notation $\e$, resp.~$\ee$, will be used for the elements of $E(X)$, resp.~of $E(\XX)$.
It follows from the assumption (v) in Definition \ref{topocrystal} that $X\backslash\Gamma\cong \XX$,
and therefore we can identify $V(\XX)$ as a subset of $V(X)$ by choosing a representative of each orbit.
Namely, since by assumption $V(\XX)=\{\xx_1,\dots,\xx_n\}$ for some $n\in\N$,
we choose $\{x_1,\dots,x_n\}\subset V(X)$ such that $\omega(x_j)=\xx_j$ for any $j\in \{1,\dots,n\}$.
For shortness we also use the notation $\check{x}:=\omega(x)\in V(\XX)$ for any $x\in V(X)$,
and reciprocally for any $\xx\in \XX$ we write $\hat{\xx}\in \{x_1,\dots,x_n\}$ for the unique element $x_j$ in this set such that $\omega(x_j)=\xx$.

As a consequence of the previous identification we can also identify $\A(\XX)$ as a subset of $\A(X)$.
More precisely, we identify $\A(\XX)$ with $\cup_{j=1}^n \A_{x_j}\subset\A(X)$ and use notations similar to the previous ones:
For any $\e\in \A(X)$ one sets $\check{\e}=\omega(\e)\in \A(\XX)$, and for any $\ee\in \A(\XX)$ one sets $\hat{\ee}\in \cup_{j=1}^n\A_{x_j}$
for the unique element in this set such that $\omega(\hat{\ee})=\ee$.
Let us stress that these identifications and notations depend only on the initial choice of $\{x_1,\dots,x_n\}\subset V(X)$.

We have now enough notations for defining the \emph{entire part} of a vertex $x$ as the map $\[\,\cdot\,\]:V(X)\to\Gamma$ satisfying
\begin{equation}\label{entirex}
\[x\]\widehat{\check{x}}=x\ .
\end{equation}
Similarly, the entire part of an edge is defined as the map $\[\,\cdot\,\]:\A(X)\to\Gamma$ satisfying
\begin{equation}\label{entiree}
\[\e\]\widehat{\check{\e}}=\e\ .
\end{equation}
The existence of the this function $[\, \cdot\,]$ follows from the assumption (v)
of Definition \ref{topocrystal} on the regularity of a topological crystal.
One easy consequence of the previous construction is that the equality
$\[\e\]=\[o(\e)\]$ holds for any $\e\in \A(X)$.

Let us finally define the map
\begin{equation*}
\eta:\A(X)\to\Gamma, \quad \eta(\e):=\[t(\e)\]\[o(\e)\]^{-1}
\end{equation*}
and call $\eta(\e)$ the \emph{index} of the edge $\e$. For any $\mu\in\Gamma$ we then infer that
\begin{equation*}
\eta(\mu\e)=\[t(\mu\e)\]\[o(\mu\e)\]^{-1}=\mu\[t(\e)\]\mu^{-1}\[o(\e)\]^{-1}=\eta(\e).
\end{equation*}
This periodicity enables us to define unambiguously $\eta:\A(\XX)\to \Gamma$ by the relation $\eta(\ee):=\eta(\hat{\ee})$
for every $\ee\in\A(\XX)$. Again, this index on $\A(\XX)$ depends only on the initial choice $\{x_1,\dots,x_n\}\subset V(X)$
and could not be define by considering only $\A(\XX)$.

We now introduce the dual group of $\Gamma$, denoted by $\hat{\Gamma}$. It consists in group homomorphisms from $\Gamma$
to the multiplicative group $\T\subset \C$ endowed with pointwise multiplication. Since $\Gamma$ is discrete,
$\hat{\Gamma}$ is a compact Abelian group and comes with a normalized Haar measure $\d\xi$ of volume 1 \cite[Proposition 4.24]{Fo95}.
We can then define the Fourier transform $\F: l^1(\Gamma)\to C(\hat{\Gamma})$ by
\begin{equation}\label{def_Fourier}
[\F f](\xi)\equiv \hat{f}(\xi):=\sum_{\mu\in \Gamma} \overline{\xi(\mu)}f(\mu)
\end{equation}
and it is well-known that this extends to a unitary map from $l^2(\Gamma)$ to $L^2(\hat{\Gamma})$
which is still denoted by $\F$. The adjoint map $\F^*:L^2(\hat{\Gamma})\to l^2(\Gamma)$ is defined on
elements in $L^1(\hat{\Gamma})$ by the formula $[\F^* u](\mu)=\int_{\hat{\Gamma}} \d \xi \;\!\xi(\mu)u(\xi)$.
Furthermore, by the Fourier inversion formula for any $f\in l^1(\Gamma)$ one has \cite[Theorem 4.21]{Fo95}:
\begin{equation*}
f(\mu)=\int_{\hat{\Gamma}}\d\xi\;\!\xi(\mu)  \hat{f}(\xi),
\end{equation*}
or equivalently for any $u \in L^1(\hat{\Gamma})$ such that $\F^* u \in l^1(\Gamma)$
\begin{equation*}
u(\xi)=\sum_{\mu\in\Gamma} \overline{\xi(\mu)} [\F^* u](\mu).
\end{equation*}

Let us now provide the direct integral decomposition mentioned at the beginning of this section.
The framework is the following: a topological crystal  $(X,\XX,\omega,\Gamma)$
and a $\Gamma$-periodic measure $m_0$ on $X$. Because of its periodicity, this measure is also well-defined on $\XX$
by the relation $m_0(\xx):=m_0(\hat{\xx})$ and $m_0(\ee):=m_0(\hat{\ee})$.
For simplicity, we keep the same notation for this measure on $\XX$.
Let us consider the Hilbert spaces $l^2(X,m_0)$ and $L^2\big(\hat{\Gamma}; l^2(\XX,m_0)\big)$, and use the shorter notation
$l^2(X)$ and $L^2\big(\hat{\Gamma}; l^2(\XX)\big)$. We also denote by $c_c(X)\subset l^2(X)$ the space of $0$-cochains of finite support.
We then define the map $\U: c_c(X) \to L^2\big(\hat{\Gamma}; l^2(\XX)\big)$ for $f\in c_c(X)$,
$\xi\in \hat{\Gamma}$, and $\xx\in V(\XX)$ by
\begin{equation}\label{def_de_U}
[\U f](\xi,\xx)=\sum_{\mu\in\Gamma}\overline{\xi(\mu)}f(\mu\hat{\xx}).
\end{equation}
Clearly, the map $\U$ corresponds the composition of two maps: the identification of $l^2(X)$ with $l^2\big(\Gamma;l^2(\XX)\big)$
and the Fourier transform introduced in \eqref{def_Fourier}.
As a consequence, $\U$ extends to a unitary map from $l^2(X)$ to $L^2\big(\hat{\Gamma};l^2(\XX)\big)$, and we shall
keep the same notation for this continuous extension.
The formula for its adjoint is then given on any $\f\in L^1\big(\hat{\Gamma};l^2(\XX)\big)$ by
\begin{equation*}
[\U^*\f](x) = \int_{\hat{\Gamma}}\d\xi\;\!\xi([x])\f(\xi,\check{x}).
\end{equation*}

\begin{Lemma}
Let  $(X,\XX,\omega,\Gamma)$ be a topological crystal and let $m_0$ be a $\Gamma$-periodic measure on $X$.
Then for any $u\in L^2\big(\hat{\Gamma};l^2(\XX)\big)$, every $\xx\in V(\XX)$ and almost every $\xi \in \hat{\Gamma}$
the following equality holds:
\begin{equation*}
[\U \Delta(X,m_0)\U^* \f](\xi,\xx)
= \sum_{\ee\in\A_{\xx}}\frac{m_0(\ee)}{m_0(\xx)} \Big[\xi\big(\eta(\ee)\big)\f\big(\xi, t(\ee)\big)- \f(\xi,\xx)\Big].
\end{equation*}
\end{Lemma}

\begin{proof}
For simplicity, we shall write $\Delta$ for $\Delta(X,m_0)$.
Let $\f\in L^2\big(\hat{\Gamma};l^2(\XX)\big)$ such that
$\U^*\f$ has a compact support on $X$. Then for almost every $\xi \in \hat{\Gamma}$ and $\xx\in V(\XX)$ one has
\begin{align*}
[\U \Delta\U^* \f](\xi,\xx) & = \sum_{\mu\in\Gamma}\overline{\xi(\mu)}[\Delta\U^* \f](\mu\hat{\xx}) \\
& =  \sum_{\mu\in\Gamma}\overline{\xi(\mu)}
\sum_{\e\in\A_{\mu\hat{\xx}}}\frac{m_0(\e)}{m_0(\mu\hat{\xx})}\Big(\big[\U^*\f\big]\big(t(\e)\big)-[\U^*\f](\mu\hat{\xx})\Big) \\
& =  \sum_{\mu\in\Gamma}\overline{\xi(\mu)}
\sum_{\e\in\A_{\hat{\xx}}}\frac{m_0(\e)}{m_0(\hat{\xx})}\Big(\big[\U^*\f\big]\big(t(\mu\e)\big)-[\U^*\f](\mu\hat{\xx})\Big) \\
& =  \sum_{\e\in\A_{\hat{\xx}}}\frac{m_0(\e)}{m_0(\hat{\xx})} \sum_{\mu\in\Gamma}\overline{\xi(\mu)}\Big(\big[\U^*\f\big]\big(t(\mu\e)\big)-[\U^*\f](\mu\hat{\xx})\Big) \\
& = \sum_{\e\in\A_{\hat{\xx}}}\frac{m_0(\e)}{m_0(\hat{\xx})} \Big[\sum_{\mu\in\Gamma}\overline{\xi(\mu)}\big[\U^*\f\big]\big(t(\mu\e)\big)- \f(\xi,\xx)\Big].
\end{align*}
By observing that
\begin{equation}\label{eq_use_ind}
t(\mu\e) = [t(\mu\e)]\widehat{\omega(t(\mu \e))} = \mu \eta(\e) \widehat{\omega(t(\e))},
\end{equation}
one infers that
\begin{align*}
[\U \Delta\U^* \f](\xi,\xx)
&=\sum_{\e\in\A_{\hat{\xx}}}\frac{m_0(\e)}{m_0(\hat{\xx})} \Big[\sum_{\mu\in\Gamma}\overline{\xi(\mu)}\big[\U^*\f\big]\big(\mu \eta(\e) \widehat{\omega(t(\e))}\big)- \f(\xi,\xx)\Big] \\
&=\sum_{\e\in\A_{\hat{\xx}}}\frac{m_0(\e)}{m_0(\hat{\xx})} \Big[\xi\big(\eta(\e)\big)\f\big(\xi, \omega(t(\e))\big)- \f(\xi,\xx)\Big] \\
&=\sum_{\ee\in\A_{\xx}}\frac{m_0(\ee)}{m_0(\xx)} \Big[\xi\big(\eta(\ee)\big)\f\big(\xi, t(\ee)\big)- \f(\xi,\xx)\Big],
\end{align*}
where for the last equality one has used that $\omega\big(t(\hat{\ee})\big)=t(\ee)$.
The statement follows then by a density argument.
\end{proof}

In order to make the connection with magnetic Laplacian, let us recall that for any $\theta:\A(\XX)\to \T$
satisfying $\theta(\overline{\ee})=\overline{\theta(\ee)}$ one defines a magnetic Laplace operator on $\XX$ by the formula
\begin{equation*}
[\Delta_\theta(\XX,m_0) \varphi](\xx)=\sum_{\ee\in\A_\xx} \frac{m_0(\ee)}{m_0(\xx)}\big(\theta(\ee)\varphi(t(\ee))-\varphi(\xx)\big) \qquad \forall \varphi\in l^2(\XX).
\end{equation*}
Thus, if for fixed $\xi \in \hat{\Gamma}$ one sets
\begin{equation}\label{def_theta_xi}
\theta_\xi: \A(\XX) \to \T, \quad \theta_\xi(\ee):= \xi\big(\eta(\ee)\big),
\end{equation}
then one infers that
\begin{equation*}
\theta_\xi(\overline{\ee})=\xi\big(\eta(\overline{\ee})\big)=\xi\big(\eta(\ee)^{-1}\big) = \overline{\xi\big(\eta(\ee)\big)} = \overline{\theta_\xi(\ee)}.
\end{equation*}
As a consequence, the operator $\Delta_{\theta_\xi}(\XX,m_0)$ defined on any $\varphi\in l^2(\XX)$ by
\begin{align*}
[\Delta_{\theta_\xi}(\XX,m_0) \varphi](\xx)& =\sum_{\ee\in\A_\xx} \frac{m_0(\ee)}{m_0(\xx)}\big(\theta_\xi(\ee)\varphi(t(\ee))-\varphi(\xx)\big) \\
& =\sum_{\ee\in\A_\xx} \frac{m_0(\ee)}{m_0(\xx)}\big(\xi\big(\eta(\ee)\big)\varphi(t(\ee))-\varphi(\xx)\big)
\end{align*}
corresponds to a magnetic Laplace operator on $\XX$.

Let us now recall that $L^2\big(\hat{\Gamma};l^2(\XX)\big) = \int_{\hat{\Gamma}}^\oplus \d\xi\;\! l^2(\XX)$.
As a consequence of the previous lemma and of the construction made above, the operator $\U \Delta(X,m_0)\U^*$ itself can be identified with the direct integral operator
$\int_{\hat{\Gamma}}^\oplus\d\xi\;\! \Delta_{\theta_\xi}(\XX,m_0)$.
In other words, the Laplace operator $\Delta(X,m_0)$ is unitarily equivalent to a direct integral of magnetic Laplace operators acting on $\XX$.

In order to get a direct integral of magnetic Schr\"odinger operators as mentioned at the beginning of this section,
it only remains to deal with the multiplication operator $R_0$ by a $\Gamma$-periodic function, as introduced in \eqref{h0vertices}.
For that purpose, let us observe that for any real $\Gamma$-periodic function defined on $V(X)$ one can associate a well-defined function
on $V(\XX)$ by the relation $R_0(\xx):=R_0(\hat{\xx})$. For simplicity (and as already done before) we keep the
same notation for this new function.
Then the following statement is obtained by a direct computation.

\begin{Lemma}
Let $R_0$ be a $\Gamma$-periodic function on $V(X)$. Then one has $\U R_0 \U^* = R_0$, or more precisely
for any $\f \in L^2\big(\hat{\Gamma};l^2(\XX)\big)$, for all $\xx\in \XX$ and a.e. $\xi \in \hat{\Gamma}$ the following equality holds:
\begin{equation*}
[\U R_0 \U^*\f](\xi,\xx)=R_0(\xx)\f(\xi,\xx).
\end{equation*}
\end{Lemma}

By adding the various results obtained in this section one can finally state:

\begin{Proposition}\label{prop_U H_0 U^*}
Let  $(X,\XX,\omega,\Gamma)$ be a topological crystal and let $m_0$ be a $\Gamma$-periodic measure on $X$.
Let $R_0$ be a real $\Gamma$-periodic function defined on $V(X)$. Then the periodic Schr\"odinger operator $H_0:=-\Delta(X,m_0)+R_0$
is unitarily equivalent to the direct integral of magnetic Schr\"odinger operators acting on $\XX$ defined by
\begin{equation*}
\int_{\hat{\Gamma}}^\oplus\d\xi\;\! \big[-\Delta_{\theta_\xi}(\XX,m_0) + R_0\big]
\end{equation*}
with $\theta_\xi$ defined in \eqref{def_theta_xi}.
\end{Proposition}

In the next section, we shall show that $H_0$ is in fact unitarily equivalent to an analytically fibered operator.

\section{Analyticity of the periodic operator}\label{section:analyticdecomposition}
\setcounter{equation}{0}

Before showing that the periodic operator $H_0$ is unitarily equivalent to an analytically fibered operator,
we shall recall a few definitions related to real analyticity as well as one version of the classical result on stratifications of Hironaka.
Doubtlessly, any reader familiar with real analyticity can skip Section \ref{subsec_anal}.
For that purpose, let us simply mention that for any topological space $\mathcal{X}$ and for any $\zeta\in \mathcal{X}$,
we shall denote by $\V_{\mathcal{X}}(\zeta)$ the set of all open neighborhoods of $\zeta$ in $\mathcal{X}$.

\subsection{A brief review of real analyticity}\label{subsec_anal}

For an open set $\UU$ in $\R^{n}$ we say that a
function $\varPhi$ defined on $\UU$, and taking values in $\R$ or $\C$, is \emph{real analytic on $\UU$} if it can be written locally as a convergent power series.
More precisely, $\varPhi$ is said to be real analytic on $\UU$ if for every $\zeta_0\in\UU$ there exists $\OO\in\V_{\UU}(\zeta_0)$ and a
(real or complex) sequence $\{a_\alpha\}_{\alpha\in\N^{n}}$ such that
\begin{equation*}
\varPhi(\zeta)=\sum_{\alpha\in\N^{n}}a_\alpha(\zeta-\zeta_0)^{\alpha}
\end{equation*}
for every $\zeta\in\OO$.
A vector-valued function is real analytic if each of its component is real analytic,
and analogously a matrix-valued function is real analytic if each of its entries is real analytic.
Clearly, real analyticity is preserved by the sum, the product, the quotient, and the composition of real analytic functions
when these operations are well-defined \cite[Propositions 2.2.2 \& 2.2.8]{KP02}.

Let us now recall that a \emph{real analytic manifold} $\MM$ of dimension $n$ is a smooth manifold such
that each transition function is real analytic. More precisely, if $\{(\OO_j,\varPhi_j)\}$ is an atlas for $\MM$,
then the maps $\varPhi_j\circ\varPhi_k^{-1}:\varPhi_k(\OO_j\cap \OO_k)\to \varPhi_j(\OO_j\cap \OO_k)$
are real analytic maps.
In this setting, a function $\Psi:\MM\to \R$ is said to be \emph{real analytic at $p\in \MM$} if for $j$ such that $p\in \OO_j$
the function $\Psi\circ \varPhi_j^{-1}$ is real analytic at $ \varPhi_j(p)$.
The function $\Psi$ is \emph{real analytic on $\MM$} if it is real analytic at every points of $\MM$.

Let us also recall the notion of semi-analytic subset and the more general notion of
subanalytic subsets. The following definitions are borrowed from sections 2 and 3 of \cite{BM}.

\begin{Definition}\label{defsemi}
Let $\MM$ be a real analytic manifold.
A subset $\SS\subset \MM$ is said to be \emph{semi-analytic} if for every $p\in \SS$ there exist $\OO\in \V_\MM(p)$
and a finite family $\{\Psi_{j\ell}\}_{j,\ell}$ of real analytic functions defined on $\OO$ such that
\begin{equation*}
\SS\cap \OO=\bigcup_{j}\bigcap_\ell \big\{x\in \OO\mid \Psi_{j\ell}(x)\bowtie_{j\ell} 0 \text{ with }\bowtie_{j\ell}\in\{>,=\}\big\}.
\end{equation*}
\end{Definition}
Let us stress  that a semi-analytic subset need not be a real analytic submanifold.

\begin{Definition}\label{defsuba}
Let $\MM$ be a real analytic manifold.
A subset $\SS\subset \MM$ is said to be \emph{subanalytic} if for every $p\in \SS$ there exist $\OO\in \V_\MM(p)$
and an additional real analytic manifold $\NN$ such that $\SS\cap \OO$ is the image of a relatively compact semi-analytic
subset of $\MM\times \NN$ under the projection onto the first factor.
\end{Definition}

In this context, the following definition of stratification can be recalled, see for example
\cite[Sec.~2]{BM} and \cite[Def.~III.1.6]{Delort}.

\begin{Definition}\label{defstrat}
A \emph{stratification} of a real analytic manifold $\MM$ is a partition
$\SSS:=\{\SS_\alpha\}_\alpha$ of $\MM$ satisfying the following conditions:
\begin{enumerate}
\item[(i)] Each $\SS_\alpha$ is a connected subset of $\MM$ and a real analytic submanifold of $\MM$,
\item[(ii)] $\SSS$ are locally finite at any point of $\MM$,
\item[(iii)] If $\SS_{\alpha}\cap\overline{\SS_{\beta}}\ne\emptyset$ then $\SS_{\alpha}\subset\overline{\SS_{\beta}}$.
\end{enumerate}
If each $\SS_\alpha$ is semi-analytic the stratification is called \emph{semi-analytic}, while if
each $\SS_\alpha$ is subanalytic the stratification is called \emph{subanalytic}.
\end{Definition}

If $\MM$ is already endowed with a locally finite family $\{\MM_j\}_j$ of subsets, one
says that the stratification $\SSS$ of $\MM$ is \emph{compatible with $\{\MM_j\}_j$} if for every $j$ and every $\alpha$
one has either $\SS_\alpha \cap \MM_j=\emptyset$ or $\SS_\alpha\subset \MM_j$.
As shown in \cite[Corol.~2.11]{BM}, given a locally finite family of semi-analytic sets on $\MM$,
there always exists a semi-analytic stratification of $\MM$ which is compatible with this family. However,
this result is not strong enough for our purpose, since one more ingredient is necessary.

\begin{Definition}
Let $\MM,\MM'$ be two real analytic manifolds, and let $f:\MM\to \MM'$ be a real analytic map.
A \emph{(semi-analytic or subanalytic) stratification} for $f$ is a pair $(\SSS,\SSS')$ of (semi-analytic or subanalytic) stratifications of $\MM$ and $\MM'$ respectively
such that for any $\SS_\alpha \in \SSS$ one has $f(\SS_\alpha)\in \SSS'$ and the rank of the Jacobian matrix of $f$ at any point of $\SS_\alpha$
is equal to the dimension of $f(\SS_\alpha)$.
\end{Definition}

We can now state the version of the theorem of stratification of Hironaka as presented in \cite[Thm.~III.1.8]{Delort},
see also \cite[Corol.~4.4]{Hardt}, \cite[Sec.~3]{H77}.
Note that we directly impose a stronger condition on $f$ since it simplifies the statement
and since this condition will be automatically satisfied in our application.

\begin{Theorem}\label{thm_Hironaka}
Let $\MM,\MM'$ be two real analytic manifolds, and let $f:\MM\to \MM'$ be a proper real analytic map.
Suppose we are given finitely many subanalytic sets $\MM_j\subset \MM$, and
finitely many subanalytic sets $\MM_k'\subset \MM'$.
Then there exists a subanalytic stratification $(\SSS,\SSS')$ of $f$
such that $\SSS$ is compatible with $\{\MM_j\}$ and $\SSS'$ is compatible with $\{\MM_k'\}$.
\end{Theorem}

\subsection{Analytic decomposition of the periodic operator}

We shall now show that $H_0$ is unitarily equivalent to an \emph{analytically fibered operator}.
We refer to \cite{GN98} and \cite[Sec.~XIII.16]{RS4} for more general information on such operators,
and restrict ourselves to the simplest framework. In that respect, the next definition is adapted to our setting.
Note that from now on we shall use the notation $\T^d$  for the $d$-dimensional (flat) torus, \ie~for $\T^d=\R^{d}/\Z^{d}$,
with the inherited local coordinates system and differential structure.
We shall also use the notation $M_n(\C)$ for the $n\times n$ matrices over~$\C$.

\begin{Definition}\label{analiticallyfibered}
In the Hilbert space $L^2(\T^d;\C^n)$, a \emph{bounded analytically fibered operator}
corresponds to a multiplication operator defined by a real analytic map $h :\T^d\to M_n(\C)$.
\end{Definition}

In order to show that the periodic operator introduced in Section \ref{sec_H_0} fits into this framework,
some identifications are necessary.
More precisely, since $\Gamma$ is isomorphic to $\Z^d$, as stated in the point (iv) of Definition \ref{topocrystal},
we know that $\hat{\Gamma}$ is isomorphic to $\T^d$. In fact, we consider that a basis of $\Gamma$ is chosen
and then identify $\Gamma$ with $\Z^{d}$, and accordingly $\hat{\Gamma}$ with $\T^d$.
As a consequence of these identifications we shall write $\xi(\mu)=e^{2\pi i\,\xi\cdot\mu}$,
where $\xi\cdot\mu = \sum_{j=1}^d\xi_j\mu_j$.
Accordingly, the Fourier transform defined in \eqref{def_Fourier} corresponds to $[\F f](\xi)\equiv \hat{f}(\xi)=\sum_{\mu\in \Z^d} e^{-2\pi i\,\xi\cdot\mu} f(\mu)$,
and its inverse to $[\F^* u](\mu)\equiv \check{u}(\mu)=\int_{\T^d}\d\xi\;\! e^{2\pi i\,\xi\cdot \mu}u(\xi)$, with $\d \xi$ the usual measure on  $\T^d$.
Note that an other consequence of this identification is the use of the additive notation for the composition of two elements of $\Z^d$,
instead of the multiplicative notation employed until now for the composition in $\Gamma$.

The second necessary identification is between $l^2(\XX)$ and $\C^n$. Indeed, since $V(\XX)= \{\xx_1,\dots,\xx_n\}$, as already mentioned in the previous section,
the vector space $l^2(\XX)$ is of dimension $n$. However, since the scalar product in $l^2(\XX)$ is defined with the measure $m_0$ while
$\C^n$ is endowed with the standard scalar product, one more unitary transformation has to be defined. More precisely,
for any $\varphi\in l^2(\XX)$ one sets
$\I:l^2(\XX)\to \C^n$ with
\begin{equation}\label{def_de_I}
\I \varphi =\big(m_0(\xx_1)^{\frac12}\varphi(\xx_1), m_0(\xx_2)^{\frac12}\varphi(\xx_2),\dots,m_0(\xx_n)^{\frac12}\varphi(\xx_n)\big).
\end{equation}
This map defines clearly a unitary transformation between $l^2(\XX)$ and $\C^n$.
Note that we shall use the same notation $\I$ for the map $L^2\big(\T^d;l^2(\XX)\big)\to L^2(\T^d;\C^n)$ acting trivially on the first variables
and acting as above on the remaining variables.

We can now state and prove the main result of this section, where we use the usual notation $\delta_{j\ell}$ for the Kronecker delta function.

\begin{Proposition}\label{prop_def_H}
Let  $(X,\XX,\omega,\Gamma)$ be a topological crystal and let $m_0$ be a $\Gamma$-periodic measure on $X$.
Let $R_0$ be a real $\Gamma$-periodic function defined on $V(X)$.
Then the periodic Schr\"odinger operator $H_0:=-\Delta(X,m_0)+R_0$
is unitarily equivalent to the bounded analytically fibered operator in $L^2(\T^d;\C^n)$
defined by the function $h_0:\T^d\to M_n(\C)$ with
\begin{equation}\label{def_h}
h_0(\xi)_{j\ell}:=-\sum_{\ee=(\xx_j,\xx_\ell)}\frac{m_0(\ee)}{m_0(\xx_j)^{\frac{1}{2}}\;\!m_0(\xx_\ell)^{\frac{1}{2}}}
\;\! e^{2\pi i\,\xi\cdot\eta(\ee)} + \big(\deg_{m_0}(\xx_j) + R_0(\xx_j)\big)\delta_{j\ell}
\end{equation}
for any $\xi \in \T^d$ and $j,\ell\in \{1,\dots,n\}$.
\end{Proposition}

\begin{proof}
The proof consists simply in computing the operator $\I \U H_0 \U^* \I^*$, and in checking that
the resulting operator is analytically fibered.
Observe first that the product $\U H_0 \U^*$ has already been computed in Proposition \ref{prop_U H_0 U^*}.
The conjugation with $\I$ is easily computed, and one directly obtains \eqref{def_h} if one takes
the equality $\xi(\mu)= e^{2\pi i\,\xi\cdot \mu}$ into account.
Since for each fixed $\mu \in \Z^d$ the map $\T^d\ni \xi \mapsto e^{2\pi i\,\xi\cdot \mu}\to \C$ is real analytic,
the matrix-valued function defined by $h_0$ is real analytic.
\end{proof}

\section{Mourre theory and the conjugate operator}\label{sec_Mourre}
\setcounter{equation}{0}

In this section we first recall some definitions related to Mourre theory, such as some
regularity conditions as well as the meaning of a Mourre estimate.
These notions will be used in the second part of the section where a conjugate operator for $H_0$
will be constructed.
Again, any reader familiar with the conjugate operator method can skip Section \ref{subsec_Mourre}
and directly start with Section \ref{subsec_A}.

\subsection{Mourre theory}\label{subsec_Mourre}

In this section we provide the strictly necessary notions for our purpose,
and refer to \cite[Sec.~7.2]{ABG96} for more information and details.

Let us consider a Hilbert space $\H$ with scalar product
$\langle\;\!\cdot\;\!,\;\!\cdot\;\!\rangle$ and norm $\|\;\!\cdot\;\!\|$. Let
also $S$ and $A$ be two self-adjoint operators in $\H$. The operator $S$ is assumed to be bounded,
and we write $\D(A)$ for the domain of $A$. The spectrum of $S$ is denoted by $\sigma(S)$ and its spectral measure by
$E_S(\;\!\cdot\;\!)$. For shortness, we also use the notation
$E_S(\lambda;\varepsilon):=E_S\big((\lambda-\varepsilon,\lambda+\varepsilon)\big)$
for all $\lambda\in\R$ and $\varepsilon>0$.

The operator $S$ belongs to $C^1(A)$ if the map
\begin{equation}\label{C1}
\R\ni t\mapsto e^{-itA}S e^{itA}\in\B(\H)
\end{equation}
is strongly of class $C^1$ in $\H$. Equivalently, $S\in C^1(A)$ if the quadratic form
\begin{equation*}
\D(A)\ni\varphi
\mapsto\langle iA\varphi,S^*\varphi\rangle -\langle iS\varphi,A\varphi\rangle \in \C
\end{equation*}
is continuous in the topology of $\H$.
In such a case, this form extends uniquely to a continuous form on $\H$, and the corresponding bounded self-adjoint operator
is denoted by $[iS,A]$.
This $C^1(A)$-regularity of $S$ with respect to $A$ is the basic ingredient for any
investigation in Mourre theory.

Let us also define some stronger regularity conditions. First of all, $S\in C^2(A)$ if the map \eqref{C1}
is strongly of class $C^2$ in $\H$. A weaker condition can be expressed as follows:
$S\in C^{1,1}(A)$ if
\begin{equation*}
\int_0^1  \frac{\d t}{t^2}\;\!\big\| e^{-itA}S e^{itA}+ e^{itA}S e^{-itA}-2S\big\|<\infty.
\end{equation*}
It is then well-known that the following inclusions hold: $C^2(A)\subset C^{1,1}(A) \subset C^1(A)$.

For any $S\in C^1(A)$, let us now introduce two subsets of $\R$ which will play a central role.
Namely, one sets
\begin{equation*}
\mu^A(S):=\big\{\lambda \in \R\mid \exists \varepsilon>0, a>0 \hbox{ s.t. } E_S(\lambda;\varepsilon)[iS,A]E_S(\lambda;\varepsilon)\geq a E_S(\lambda;\varepsilon) \big\}
\end{equation*}
as well as the larger subset of $\R$ defined by
\begin{align*}
\tilde{\mu}^A(S):=\big\{\lambda \in \R\mid & \exists \varepsilon>0, a>0, K\in \K(\H) \hbox{ s.t. }  \\
& \quad E_S(\lambda;\varepsilon)[iS,A]E_S(\lambda;\varepsilon)\geq a E_S(\lambda;\varepsilon)+K \big\}.
\end{align*}

In order to state one of the main results in Mourre theory, let us still set
$\KK:=\big(\D(A),\H\big)_{\frac{1}{2},1}$ for the Banach space obtained by real interpolation.
We refer to \cite[Sec.~3.4]{ABG96} for more information about this space and for a general
presentation of Besov spaces associated with the pair $\big(\D(A),\H\big)$.
Since $\B(\H)\subset \B(\KK,\KK^*)$, for any $z\in \C\setminus \R$ the resolvent $(S-z)^{-1}$ of $S$ belongs to
these spaces, and the following extension holds:

\begin{Theorem}[{\cite[Theorem 7.3.1.]{ABG96}}]\label{mourreabg}
Let $S$ be a self-adjoint element of $\B(\H)$ and assume that $S\in C^{1,1}(A)$.
Then the holomorphic function $\C_{\pm}\ni z\to (S-z)^{-1}\in B(\KK,\KK^{*})$
extends to a weak$^{*}$ continuous function on $\C_\pm\cup \mu^A(S)$.
\end{Theorem}

Let us still mention how a perturbative scheme can be developed.
Consider a ``perturbation'' $V\in \K(\H)$ and assume that $V$ is self-adjoint and belongs to $C^{1,1}(A)$ as well.
Even if $\mu^A(S)$ is known, it usually quite difficult to compute the corresponding set
$\mu^A(S+V)$ for the self-adjoint operator $S+V$.
However, the set $\tilde{\mu}^A(S)$ is much more stable since $\tilde{\mu}^A(S)=\tilde{\mu}^A(S+V)$,
as a direct consequence of \cite[Thm.~7.2.9]{ABG96}.

Based on this observation, the following adaptation of \cite[Thm.~7.4.2]{ABG96} can be stated in our context:

\begin{Theorem}\label{mourreap}
Let $S$ be a self-adjoint element of $\B(\H)$ and assume that $S\in C^{1,1}(A)$.
Let $V\in \K(\H)$ and assume that $V$ is self-adjoint and belongs to $C^{1,1}(A)$.
Then, for any closed interval $I\subset \tilde{\mu}^A(S)$ the operator $S+V$ has at most
a finite number of eigenvalues in $I$, and no singular continuous spectrum in $I$.
\end{Theorem}

Let us finally mention that under additional condition on the perturbation $V$,
information on the local wave operators can be deduced. We shall come back on
this topic later on.

\subsection{The conjugate operator}\label{subsec_A}

In this section, we construct a conjugate operator for a self-adjoint bounded analytically fibered operator $h$ in $L^2(\T^d;\C^n)$.
At the end of the day, the operator $h$ will be the operator $h_0$ introduced in Proposition \ref{prop_def_H},
but we prefer to provide an abstract construction. Note that the following content
is inspired from an analog construction of \cite{GN98}. However, our setting is slightly simpler,
and in addition we provide here much more details.

Let us recall that a self-adjoint bounded analytically fibered operator corresponds to a multiplication operator
by a real analytic function $h:\T^d\to M_n(\C)$ with $h(\xi)$ Hermitian for any $\xi\in \T^d$.
For consistency, the multiplication operator will also be denoted by $h$.
For such an operator we introduce some notations.
For any Borel set $\V\subset\R$ and any $\xi\in\T^{d}$, let us denote by $\pi_{\V}(\xi)$ the spectral projection $E_{h(\xi)}(\V)$,
\ie~the projection in $\C^{n}$ onto the vector space generated by eigenvectors associated with the eigenvalues of $h(\xi)$ that lie in $\V$.
We also recall that $\sigma\big(h(\xi)\big)$ denotes the set of eigenvalues of $h(\xi)$.
Furthermore, we set:
\begin{itemize}
\item $\Sigma:=\big\{(\lambda,\xi)\in\R\times\T^d,\lambda\in\sigma\big(h(\xi)\big)\big\}$ ,
\item $\mul:\R\times\T^d\to\N$ defined by $(\lambda,\xi)\to\dim \pi_{\{\lambda\}}(\xi)\C^n$ ,
\item $\Sigma_j:=\{(\lambda,\xi)\in\R\times\T^d,\mul(\lambda,\xi)=j\}$ for any $j\in\{0,1\dots,n\}$.
\end{itemize}

The set $\Sigma$ is called the \emph{Bloch variety} (or the set of energy-momentum) of $h$ and will be the central object of this section.
We also denote by $p_{\R}:\Sigma\to\R$ and $p_{\T^d}:\Sigma\to\T^{d}$ the projection on each coordinate of $\Sigma$.
Some properties of $h$ and the above related objects are gathered in the next lemma.
We also refer to {\cite[Lemma 3.4]{GN98}} for a similar statement in a more general setting.

\begin{Lemma}\label{vecindad}
The application $\mul:\R\times \T^d\to\N$ is upper semicontinuous.
Furthermore, for all $(\lambda_0,\xi_0)\in\R\times \T^d$, there exist an interval $I_0\in\V_\R(\lambda_0)$ and
$\TT_0\in \V_{\T^d}(\xi_0)$ such that:
\begin{enumerate}
\item[(i)] $\pi_{I_0}(\xi_0)=\pi_{\{\lambda_0\}}(\xi_0)$,
\item[(ii)] The map $\xi\to \pi_{I_0}(\xi)\in M_n(\C)$ is real analytic in $\TT_0$.
\end{enumerate}
\end{Lemma}

Before providing the proof we want to stress that the theory of hyperbolic polynomials allows us to show that the eigenvalues
behave well on $\xi$, and this will be used to choose some convenient neighborhoods.
More precisely, for $h$ as above, the eigenvalues of $h(\xi)$ are given by the roots of $\delta(\lambda,\xi):=\det\big(\lambda \Id_n-h(\xi)\big)$.
Since each entry of the matrix $h(\xi)$ is real analytic as function of $\xi$, $\delta(\lambda,\xi)$ can be written as follows:
\begin{equation}\label{hyperbolic}
\delta(\lambda,\xi)=\det\big(\lambda \Id_n-h(\xi)\big)=\lambda^{n}+\sum_{j=1}^{n}a_{n-j}(\xi)\lambda^{n-j}
\end{equation}
where each function $a_{n-j}$ is real analytic because it is the product of finitely many real analytic functions.
Let us denote by $\{\lambda_1(\xi),\dots,\lambda_n(\xi)\}$ the family of eigenvalues of $h(\xi)$ that correspond to the roots
of \eqref{hyperbolic}. Then, it can be shown that the map $\xi\to\big(\lambda_1(\xi),\dots,\lambda_n(\xi)\big)\in\R^{n}$
is locally Lipschitz \cite[Theorem 4.1]{KP08}.

\begin{proof}[Proof of Lemma \ref{vecindad}]
Let us fix $(\lambda_0,\xi_0)\in\R\times\T^{d}$. It is clear that if $\lambda_0$ is not an eigenvalue of $h(\xi_0)$,
then both conditions hold trivially since we can find $I_0$ and $\TT_0$ such that $I_0\cap\sigma\big(h(\xi)\big)=\emptyset$ for every $\xi\in\TT_0$.

Suppose now that $\lambda_0$ is an eigenvalue of $h(\xi_0)$. We choose $I_0$ such that its closure contains no other eigenvalue of $h(\xi_0)$,
which implies in particular that $\pi_{\{\lambda_0\}}(\xi_0)=\pi_{I_0}(\xi_0)$.
In fact, by choosing an interval $I_0=(a_0,b_0)$ small enough, we can also choose a neighborhood $\TT_0$ of $\xi_0$ such that
for any $\xi\in \TT_0$ we have $\sigma\big(h(\xi)\big)\cap\{a_0,b_0\}=\emptyset$.
Around $I_0$ we choose a positively oriented closed curve $\Gamma_0$ in $\C$, sufficiently close to $I_0$ such that it does not intersect
the spectrum of $h(\xi)$ for every $\xi\in \TT_0$. Hence, for every $\xi\in\TT_0$, the eigenvalues of $h(\xi)$ that lay inside $\Gamma_0$
correspond to $\lambda_0$, or more precisely if $\lambda_j(\xi)$ lies inside $\Gamma_0$ we have $\lambda_j(\xi_0)=\lambda_0$.

As a consequence of this construction it follows that
\begin{equation}\label{resolvent}
\pi_{I_0}(\xi)=\frac{1}{2\pi i}\oint_{\Gamma_0}\,dz\,\big(z-h(\xi)\big)^{-1}.
\end{equation}
Finally, since $(z,\xi)\to\big(z-h(\xi)\big)^{-1}$ is analytic in the two variables on any domain in which $z$ is not equal to any eigenvalues of $h(\xi)$,
as shown for example in \cite[Thm II.1.5]{Ka95}, we infer from \eqref{resolvent} that the map $\xi\to\pi_{I_0}(\xi)$ is real analytic.

We now recall that a real valued function defined on a topological space $\mathcal{X}$ is said to be upper semicontinuous at $x_0$
if for every $\epsilon>0$ there exists $\UU\in\V_\mathcal{X}(x_0)$ such that $\sup_{x\in\UU}f(x)\leq f(x_0)+\epsilon$.
If we pick $I_0\times \TT_0$ as neighborhood of $(\lambda_0,\xi_0)$ we have for $(\lambda,\xi)\in I_0\times \TT_0$ that
\begin{equation}\label{mulsemi}
\mul(\lambda,\xi)= \dim\pi_{\{\lambda\}}(\xi)\C^n\leq \dim\pi_{I_0}(\xi)\C^n= \dim\pi_{I_0}(\xi_0)\C^n= \dim\pi_{\{\lambda_0\}}(\xi_0)\C^n,
\end{equation}
where $\dim\pi_{I_0}(\xi)\C^n= \dim\pi_{I_0}(\xi_0)\C^n$ is due to the analyticity of the map $\xi\to\pi_{I_0}(\xi)$.
\end{proof}

The first step towards the construction of the conjugate operator is to provide a stratification of the Bloch variety.
The following proposition will enable us to derive it from Theorem \ref{thm_Hironaka}.
Before its statement, observe that $\R\times \T^n$ is a $(n+1)$-dimensional real analytic manifold.

\begin{Proposition}\label{sigmasemi}
$\{\Sigma_j\}_{j=1}^{n}$ is a family of semi-analytic sets in $\R\times\T^d$.
\end{Proposition}

\begin{proof}
For any $(\lambda_0,\xi_0)\in\R\times\T^d$ we set $\OO=I_0\times \TT_0\in\V_{\R\times\T^d}(\lambda_0,\xi_0)$  as in Lemma \ref{vecindad}.
Then, for every $j>\mul(\lambda_0,\xi_0)$ we have $\Sigma_j\cap\OO=\emptyset$ by \eqref{mulsemi}, so we only need to consider $j\le\mul (\lambda_0,\xi_0)$.
Let us also recall that $\delta(\lambda,\xi)=\det\big(\lambda\Id_n-h(\xi)\big)$. By the discussion after the statement of Lemma \ref{vecindad},
$\delta$ admits real analytic derivatives on each variable. In addition, $\Sigma_j\cap\OO$ is described as follows:
\begin{align*}
\Sigma_j\cap\OO = &\,\big\{(\lambda,\xi)\in\OO \mid \lambda\text{ is an eigenvalue of multiplicity }j\text{ of }h(\xi)\big\}\\
= &\,\Big\{(\lambda,\xi)\in\OO\mid\delta(\lambda,\xi)
=\frac{\partial\delta}{\partial\lambda}(\lambda,\xi)=\cdots=\frac{\partial^{j-1}\delta}{\partial\lambda^{j-1}}(\lambda,\xi)=0,
\frac{\partial^{j}\delta}{\partial\lambda^{j}}(\lambda,\xi)\ne0\Big\}.\notag
\end{align*}
Then we deduce from Definition \ref{defsemi} that each $\Sigma_j$ is semi-analytic in $\R\times\T^d$.
\end{proof}

We have just shown that $\{\Sigma_j\}_{j=0}^n$ is a finite family of semi-analytic subsets of $\R\times\T^{d}$.
Since $p_{\R}:\R\times\T^d\to\R$ is proper and real analytic we can apply Theorem \ref{thm_Hironaka} to get a stratification $(\SSS,\SSS')$ of $p_{\R}$
such that $\SSS$ is compatible with $\{\Sigma_j\}_{j=1}^{n}$.
We recall that each $\SS_\alpha\in\SSS$ is contained in only one $\Sigma_j$ and that $\SSS'$ is a stratification of $\R$.
We will denote by $\tau $ \emph{the set of thresholds}, and this set is given by the union of the elements of dimension $0$ of $\SSS'$.
The thresholds are the levels of energy where one can not construct a conjugate operator.

\begin{Definition}\label{Tau}
Let $h$ be a real analytic function $\T^d\to M_n(\C)$ with $h(\xi)$ Hermitian for any $\xi \in \T^d$.
\emph{The set of thresholds} $\tau\equiv \tau(h)$ is defined by
\begin{equation*}
\tau:=\bigcup_{\dim \SS'_\beta=0}\SS'_\beta\ ,
\end{equation*}
where $\SSS'=\{\SS'_\beta\}_\beta$ is the partition of $\R$ given by Theorem \ref{thm_Hironaka} applied
to the proper real analytic function $p_\R$ and the family of semi-analytic subsets $\{\Sigma_j\}_{j=1}^n$.
\end{Definition}

Note that $\tau$ is a discrete subset of $\R$ because $\SSS'$ is locally finite,
\ie~only a finite numbers of $\SS'_\beta$ intersects the neighborhood of a given $\lambda\in\R$.
It is also easily observed that $\tau$ contains the energy levels corresponding to \emph{flat bands},
\ie~a value $\lambda\subset \R$ satisfying $\lambda_j(\xi)=\lambda$ for all $\xi$ and some fixed $j\in \{1,\dots,n\}$.

We start now the construction of the conjugate operator for a fixed closed interval $I\subset\R\backslash\tau$.
This is done in three steps: first we construct $A_{\lambda_0,\xi_0}$ for fixed $\lambda_0\in I$ and $\xi_0\in\T^{d}$;
then we sum over all the eigenvalues $\lambda$ of $h(\xi_0)$ that lie in $I$ and obtain $A_{\xi_0}$;
finally we define $A_I$ by smoothing a finite family of such $A_{\xi_0}$.

Let $(\lambda_0,\xi_0)$ be fixed with $\lambda_0\in I$. We denote by $\OO$ the neighborhood of $(\lambda_0,\xi_0)$
constructed as in Lemma \ref{vecindad}, \ie~$\OO=I_0\times\TT_0$.
Then $(\lambda_0,\xi_0)\in\SS_\alpha\subset\Sigma_j$ for a unique $\alpha$.
Without loss of generality we can assume that $\Sigma_j\cap\OO=\SS_\alpha\cap\OO$.
Let $s$ denote the dimension of the submanifold $\SS_\alpha$.
Furthermore, since $p_{\T^d}|_{\SS_\alpha}$ is injective the subset $p_{\T^d}(\SS_\alpha\cap\OO)\subset \T^d$ has also dimension $s$.
This enables us to find a neighborhood $\WW_0$ of the identity in $\R^{d}$ diffeomorphic to $\TT_0$, or more precisely there exists a diffeomorphism
\begin{equation}\label{eq_iota}
\iota_0:\TT_0\to\WW_0 \qquad \hbox{with}\qquad
\iota_0\big(p_{\T^{d}}(\SS_\alpha\cap\OO)\big)\subset\R^{s}\times\mathbf{0},
\end{equation}
see for example \cite[Theorem 2.10.(2)]{S99}.
Let us then set $x=(x',x'')\in\WW_0$ with $x'\in\R^{s}$ and $x''\in\R^{d-s}$.
We also define $f:I_0\times\WW_0\to\R$ by
\begin{equation*}
f(\lambda,x):=\frac{\partial^{j-1}\delta}{\partial\lambda^{j-1}}\big(\lambda,\iota_0^{-1}(x)\big)\ .
\end{equation*}
It follows from the proof of Proposition \ref{sigmasemi} that $f(\lambda,x',0)=0$ and $\frac{\partial f}{\partial\lambda}(\lambda,x',0)\neq0$ if $\lambda$ is such that
$\big(\lambda,\iota_0^{-1}(x)\big)\in\SS_\alpha$. By the implicit function theorem as for example presented in \cite[Theo. 2.3.5.]{KP02} and
maybe in a smaller subset $\WW_0$, we get that there exists a real analytic function $\bl:\WW_0\to\R$ such that
$f\big(\bl(x),x\big)=0$ for every $x\in\WW_0$. Then we have
\begin{equation}\label{parametrossa}
\SS_\alpha\cap\OO=\Bigl\lbrace\bigl(\boldsymbol{\lambda}(x',0),\iota_0^{-1}(x',0)\bigr) \mid (x',0)\in\WW_0\Bigr\rbrace .
\end{equation}

Let us denote by $(\iota_0^{-1})^{*}$ the pullback by $\iota_0^{-1}$ defined for $\varphi$ with support on $\TT_0$ and for any $x\in\WW_0$ by
$[(\iota_0^{-1})^{*}\varphi](x)=\varphi\big(\iota_0^{-1}(x)\big)$.
Analogously the pullback $\iota_0^{*}$ is defined by $[\iota_0^{*}g](\xi)=g\big(\iota_0(\xi)\big)$ for any $g$ defined on $\WW_0$.
We denote by $D_j=-i\partial_j$ the operator of differentiation with respect to the $j-$variable in $\R^{d}$.
We also set $\partial^{(s)}=(\partial_1,\dots,\partial_s)$ and $D^{(s)}=(D_1,\dots,D_s)$.
If we keep the notation $\pi_{I_0}$ for the matrix-valued multiplication operator
acting on $L^2(\T^d;\C^n)$ we can define $A_{\lambda_0,\xi_0}$ on $C^\infty_c(\TT_0;\C^{n})\subset L^2(\T^d;\C^n)$ by
\begin{equation*}
A_{\lambda_0,\xi_0}:=\tfrac{1}2\pi_{I_0}\iota_0^{*}\[(\partial^{(s)}\bl)\cdot D^{(s)}+D^{(s)}\cdot(\partial^{(s)}\bl)\](\iota_0^{-1})^{*}\pi_{I_0}.
\end{equation*}

By repeating this construction for each eigenvalue $\lambda_j$ of $h(\xi_0)$ lying in $I$ we can define
\begin{equation}\label{Axi0}
A_{\xi_0}:=\sum\limits_{\lambda_j\in\sigma({h(\xi_0)})\cap I}A_{\lambda_j,\xi_0}.
\end{equation}
It follows that for every $\xi_0 \in \T^{d}$ we can find a neighborhood $\TT_0$, given by the intersection of the neighborhoods constructed
for each pair $(\lambda_j,\xi_0)$, and an operator $A_{\xi_0}$ defined by \eqref{Axi0} on $C^\infty_c(\TT_0;\C^n)$.

We now define $\UU_I:=p_{\T^d}(p_\R^{-1}(I))$. Since we chose $I$ closed, $\UU_I$ is compact.
We can then consider finitely many pairs $(\xi_\ell,\TT_\ell)$ such that $A_{\xi_\ell}$ acts on $C^\infty_c(\TT_\ell;\C^{n})$
and such that $\UU_I\subset\bigcup \TT_\ell$. Considering a smooth partition of unity on $\T^{d}$, we can find a family of
smooth functions $\chi_\ell$ satisfying $\sum \chi_\ell^{2}(\xi)=1$ for $\xi\in\UU_I$ and such that each $\chi_\ell$
has support contained in $\TT_\ell$. The candidate for our conjugate operator is then given by
\begin{equation}\label{ai}
A_I=\sum_\ell \chi_\ell A_{\xi_\ell}\chi_\ell
\end{equation}
and is defined on $C^\infty(\T^d;\C^n)$. Note that $A_I$ depends on the covering $\{\TT_\ell\}$ of $\UU_I$ and we will impose later on another condition on
this covering to ensure the positivity of the commutator of $[ih,A_I]$ once suitably localized.

In order to further analyze this operator, let us recall that there exist at least three complementary ways
of considering operators acting on $\T^d$. For example, one can take advantage of the group structure of $\T^d$
(with dual group $\Z^d$) and develop a pseudodifferential calculus in this context.
A related approach consists in using a periodic version of the usual pseudodifferential operator of $\R^d$,
as already sketched in the Appendix of \cite{GN98a}.
Since $\T^d$ is also a compact smooth manifold, a more geometrical approach can be used.
Note that these various approaches and their relations have been thoroughly studied in \cite[Chap.~3-5]{RT09}.
We provide in the next paragraphs a few essential definitions or results, and refer to this reference for more information.

Let us first recall that for any $s\in \R$ the Sobolev space $\H^s(\T^d)$ consists in
the space of distributions $u\in \D'(\T^d)$ such that $\|u\|_{\H^s(\T^d)}$ is finite, with
\begin{equation*}
\|u\|_{\H^s(\T^d)}:= \Big(\sum_{\mu\in\Z^d} (1+|\mu|^2)^{s} |\check{u}(\mu)|^2 \Big)^{1/2}
\end{equation*}
and with $|\mu|$ the inherited Euclidean norm on $\Z^d$.
Note that the use of $\check{u}$ instead of the more conventional notation $\hat{u}$
comes from our initial choice of the Fourier transform, from $l^2(\Z^d)$ to $L^2(\T^d)$.
We also recall that the Fourier transform is a bijective map between the Schwartz space $\S(\Z^d)$
and the space $C^\infty(\T^d)$.
In analogy with pseudodifferential operators acting on $\R^d$ the toroidal pseudodifferential operators
are then defined by the formula
\begin{equation}\label{eq_pseudo}
[\Op(a)u](\xi)=\sum_{\mu\in \Z^d} e^{-2\pi i\xi\cdot\mu} \;\! a(\xi,\mu)\;\! \check{u}(\mu)
\end{equation}
for suitable symbol $a:\T^d\times \Z^d\to \C$ and any functions $u\in C^\infty(\T^d)$.
For example, for any $m\in \R$ a convenient class of symbols $S^m(\T^d\times \Z^d)$ is defined by
those functions $a:\T^d\times \Z^d$ which are smooth in the first variables and which satisfy
\begin{equation}\label{eq_sym}
\big|\triangle^\alpha \partial^\beta a(\xi,\mu)\big|\leq  c\;\! \langle \mu\rangle^{m-|\alpha|_1}
\qquad \forall \xi \in \T^d, \ \mu\in \Z^d
\end{equation}
for some constant $c$ which depend on the symbol $a$, on $\alpha,\beta\in \N^d$ and on $m$.
In \eqref{eq_sym} we have used the notations $\langle \mu\rangle^s$ for $(1+|\mu|^2)^{s/2}$, $\partial$ for the differentiation with respect to
the $\xi$-variable, and $\triangle$ for the difference operator defined on $f:\Z^d\to \C$ by
\begin{equation*}
[\triangle_j f](\mu) = f(\mu+\delta_j)-f(\mu)\equiv f\big(\mu_1,\dots,\mu_{j-1},\mu_j+1,\mu_{j+1},\dots,\mu_d\big)-f(\mu).
\end{equation*}
For any $\alpha \in \N^d$ we have also used the notations
$\triangle^\alpha =\prod_{j=1}^d (\triangle_j)^{\alpha_j}$
and $|\alpha|_1$ for $\sum_{j=1}^d \alpha_j$.

As mentioned before, another convenient approach consists in considering $\T^d$ as a smooth manifold
and by defining differentiable operators through localizations.
In this framework an operator $A:C^\infty(\T^d)\to C^\infty(\T^d)$ is a differential operator of order $\ell\geq 0$
if for any chart $(\OO,\varPhi)$ (also called local coordinates) one has
\begin{equation*}
A = \varPhi^* \sum_{|\alpha|_1\leq \ell} b_\alpha\;\! D^\alpha (\varPhi^{-1})^*
\end{equation*}
where the notation for the pullback has been used again, and where $b_\alpha$ are multiplication operators by smooth functions on $\R^d$.
We denote the class of such differential operators of order $\ell$ by $\mathrm{Diff}^\ell(\T^d)$.
These operators are special instances of the more general set of pseudodifferential operators of order $\ell$ on $\T^d$,
denoted by $\Psi^\ell(\T^d)$.

Our interest in having recalled these frameworks relies in the following two results: Firstly,
any operators $A\in \Psi^\ell(\T^d)$ is equal to $\Op(a)$ for some $a\in S^\ell(\T^d\times \Z^d)$,
as shown in a much more general context in \cite[Thm.~5.4.1]{RT09}.
Secondly, if $a\in S^m(\T^d\times \Z^d)$, then $\Op(a)$ extends to a bounded linear operator from $\H^s(\T^d)$
to $\H^{s-m}(\T^d)$ for every $s\in \R$ \cite[Prop.~4.2.3]{RT09}.
Now, by performing a tensor product of the spaces constructed above with the matrices $M_n(\C)$, one directly
infers that the operator $A_I$ introduced in \eqref{ai} is a differential operator of order $1$
and that this operator extends to a bounded operator from $\H^s(\T^d;\C^n)$
to $\H^{s-1}(\T^d;\C^n)$ for every $s\in \R$.

We have now introduced enough material for providing a simple proof of the following statement.

\begin{Lemma}\label{lem_ess_adj}
The operator $A_I$ defined in \eqref{ai} is essentially self-adjoint on $C^\infty(\T^d;\C^n)$.
\end{Lemma}

\begin{proof}
This proof is based on an application of Nelson's commutator theorem, as presented in \cite[Thm.~X.37]{RS75}.
For its application, we denote by $\Delta_{\T_d}$ the Laplace operator on $L^2(\T^d)$,
which is a self-adjoint operator with domain $\H^2(\T^d)$ and which is essentially self-adjoint on $C^\infty(\T^d)$.
We also set $\Lambda:=\big(\Id-\Delta_{\T^d}\big)\otimes \Id_n$ which is now a self-adjoint operator in $L^2(\T^d;\C^n)$ with domain
$\H^2(\T^d;\C^n)$ and which is essentially self-adjoint on $C^\infty(\T^d;\C^n)$.
Let us note that these operators can also be seen as second order differential operators on $[0,1]^d$
with periodic boundary conditions.

Now, it is easily observed that $A_I$ is symmetric on $C^\infty(\T^d;\C^n)$.
In addition, since $A_I$ extends to a bounded operator from $\H^1(\T^d;\C^n)$
to $\H^{0}(\T^d;\C^n)\equiv L^2(\T^d;\C^n)$, as mentioned before the statement,
this operator is \emph{a fortiori} bounded from $\H^2(\T^d;\C^n)$ to $L^2(\T^d;\C^n)$.
As a consequence, one infers that there exist $c,c'>0$ such that for any $f\in C^\infty(\T^d;\C^n)$
\begin{equation*}
\|A_I f\|_{L^2(\T^d;\C^n)}\leq c\;\!\|f\|_{\H^2(\T^d;\C^n)} \|\leq c'\;\!\|\Lambda f\|_{L^2(\T^d;\C^n)}\ .
\end{equation*}
We refer also to \cite[Rem.~4.8.4]{RT09} for the second inequality.
Then, either from a direct computation performed on $C^\infty(\T^d;\C^n)$ or from an application of the abstract result
\cite[4.7.10]{RT09} one deduces that the commutator $[A_I,\Lambda]$ corresponds to a differential operator of order $2$.
It thus follows that there exists $c>0$ such that for any $f\in C^\infty(\T^d;\C^n)$ one has
\begin{equation*}
\big|\langle A_I f,\Lambda f\rangle - \langle \Lambda f,A_I f\rangle \big| \leq c\;\! \langle f,\Lambda f\rangle = c\;\!\|\Lambda^{1/2}f\|^2 .
\end{equation*}
The statement of the lemma follows then from the mentioned Nelson's commutator theorem.
\end{proof}

We are now in a suitable position for proving a Mourre estimate,
or in other words the positivity of $[ih,A_I]$ when suitably localized.
As mentioned at the beginning of this section, a similar result already appeared in \cite[Thm.~3.1]{GN98},
but the above construction and the following proof have been adapted to our context.

\begin{Theorem}\label{hcinfinity}
Let $h$ be a real analytic function $\T^d\to M_n(\C)$ with $h(\xi)$ Hermitian for any $\xi \in \T^d$,
and let also $h$ denote the corresponding multiplication operator in $L^2(\T^d;\C^n)$.
Let $\tau$ be the set of thresholds provided by Definition \ref{Tau} and let $I$ be any closed interval in $\R\setminus \tau$.
Then, there exist a finite family of pairs $\{(\TT_\ell, \xi_\ell)\}$ with $\xi_\ell\in\TT_\ell$
such that for the operator $A_I$ defined by \eqref{ai} the following two properties hold:
\begin{enumerate}
\item[(i)] the operator $h$ belongs to $C^2(A_I)$,
\item[(ii)] there exists a constant $a_I>0$ such that
\begin{equation}\label{eq_Mourre}
E_h(I)\[ih,A_I\]E_h(I)\ge a_I E_h(I)\ .
\end{equation}
\end{enumerate}
\end{Theorem}

Before providing the proof, let us restate part of the previous statement with the notations introduced in Section \ref{subsec_Mourre}.
As a consequence of \eqref{eq_Mourre}, for any closed interval $I\equiv [a,b]\subset \R\setminus \tau$, one has
\begin{equation}\label{eq_mu}
(a,b)\subset \mu^{A_I}(h)\subset \tilde\mu^{A_I}(h).
\end{equation}

\begin{proof}
Let $(\lambda_0,\xi_0)\in \T^d\times \R$ be fixed with $\lambda_0\in I$, and let $\iota_0$ be the associated diffeomorphism introduced in \eqref{eq_iota}.
For shortness we also set $\pi_0:=\pi_{I_0}$,  $\tilde \lambda_0 :=\iota_0^* \boldsymbol{\lambda}(\iota_0^{-1})^{*}$, $\nabla_{0}=\iota^{*}_0 D^{(s)}(\iota_0^{-1})^{*}$
and $\partial_{0}=\iota^{*}_0 \partial^{(s)}(\iota_0^{-1})^{*}$.
With these notations one has
\begin{align*}
A_{\lambda_0,\xi_0} &=\tfrac{1}2\pi_{I_0}\iota_0^{*}\[(\partial^{(s)}\bl)\cdot D^{(s)}+D^{(s)}\cdot(\partial^{(s)}\bl)\](\iota_0^{-1})^{*}\pi_{I_0} \\
& = \tfrac{1}2\pi_{0}\[(\partial_0\tilde{\lambda}_0)\cdot \nabla_0 +\nabla_0 \cdot(\partial_0\tilde{\lambda}_0)\]\pi_{0} \\
& =\pi_0\big((\partial_0\tilde{\lambda}_0)\cdot\nabla_0\big)\pi_0+\tfrac{i}{2}\pi_0(\Delta_0\tilde{\lambda}_0)\pi_0
\end{align*}
where $-\Delta_0 := \iota_0^*\big(\sum_{j=0}^s\partial_j^2\big)(\iota_0^{-1})^*$.

Now, since both operators $h$ and $A_{\lambda_0,\xi_0}$ leave $C^\infty(\TT_0;\C^n)$ invariant, the commutator $[ih,A_{\lambda_0,\xi_0}]$
can be defined as an operator on $C^{\infty}(\TT_0;\C^{n})$. On this set one has
\begin{equation*}
[ih,A_{\lambda_0,\xi_0}]=[ih,\pi_0\big((\partial_0\tilde{\lambda}_0)\cdot\nabla_0\big)\pi_0] - \tfrac{1}{2}[h,\pi_0(\Delta_0\tilde{\lambda}_0)\pi_0]
\end{equation*}
Note also that the second term in the r.h.s.~vanishes since $\Delta_0\tilde{\lambda}_0$ is scalar and since $h$ commutes with $\pi_{0}$.
Furthermore we have for $\varphi\in C^\infty(\TT_0;\C^n)$ that
\begin{align*}
& \Big(\big[ih,\pi_0\big((\partial_0\tilde{\lambda}_0)\cdot\nabla_0\big)\pi_0\big]\varphi\Big)(\xi) \\
& =i h(\xi)\pi_{0}(\xi)(\partial_0\tilde{\lambda}_0)(\xi)\cdot\Bigl((\nabla_0\pi_{0})(\xi)\pi_{0}(\xi)\varphi(\xi)+\pi_{0}(\xi)\bigl(\nabla_0(\pi_{0}\varphi)\bigr)(\xi)\Bigr) \\
& \quad - i\pi_{0}(\xi)(\partial_0\tilde{\lambda}_0)(\xi)\cdot\Bigl(\big(\nabla_0(\pi_{0}h)\big)(\xi)\pi_{0}(\xi)\varphi(\xi)
+\pi_{0}(\xi)h(\xi)\big(\nabla_0(\pi_{0}\varphi)\big)(\xi)\Bigr)\ .
\end{align*}
Since $h$ commutes with each (scalar) component of $\partial_0\tilde{\lambda}_0$ the second terms of the parenthesis cancel each others.
Consequently, one infers that $[h,iA_{\lambda_0,\xi_0}]$ corresponds to a bounded fibered operator $B_{\lambda_0,\xi_0}$ with its fibers defined by
\begin{equation*}
b_{\lambda_0,\xi_0}(\xi)=i \pi_{0}(\xi)(\partial_0\tilde{\lambda}_0)(\xi)\cdot\Bigl(h(\xi)(\nabla_0\pi_{0})(\xi)-\bigl(\nabla_0(\pi_{0}h)\bigr)(\xi)\Bigr)\pi_{0}(\xi)\ .
\end{equation*}
The first term in the parenthesis vanishes because $\pi(\cdot) \pi'(\cdot) \pi(\cdot)= 0$ for any differentiable family of projections.
For the second term one has by construction $\pi_{0}(\xi)h(\xi)=\tilde{\lambda}_0(\xi)\pi_{0}(\xi)$ for $\xi\in \TT_0$, and therefore
\begin{equation*}
b_{\lambda_0,\xi_0}(\xi)=
i\pi_{0}(\xi)(\partial_0\tilde{\lambda}_0)(\xi)\cdot\big(\nabla_0(\tilde{\lambda}_0\pi_{0})\big)(\xi)\pi_{0}(\xi)
=\pi_{0}(\xi)|(\partial_0\tilde{\lambda}_0)(\xi)|^2\pi_{0}(\xi)\ .
\end{equation*}

We now recall that by the definition of the set of thresholds $\tau$ and the properties of the stratification
one has $\dim(p_\R|_{\SS_\alpha})=1$ with $\SS_\alpha$ the real analytic submanifold of $\R\times \T^d$ with $(\lambda_0,\xi_0)\in \SS_\alpha$.
Combining this with \eqref{parametrossa} we have that
\begin{equation*}
1=\dim(p_\R|_{\SS_\alpha})=\dim\big(\boldsymbol{\lambda}(\WW_0,0)\big)=\rank(\partial_0\tilde{\lambda}_0)
\end{equation*}
from which we deduce that $\partial_0\tilde{\lambda}_0$ does not vanish on $\TT_0$.
We get then $b_{\lambda_0,\xi_0}(\xi_0)\ge c_{0,0} \pi_{I_0}(\xi_0)$, with $c_{0,0}>0$, and since for fixed $\xi_0$ there are at most $n$ constants we infer
\begin{equation}\label{estimacionenxi0}
b_{\xi_0}(\xi_0):=\sum_{\lambda_i\in\sigma({h(\xi_0)})\cap I} b_{\lambda_i,\xi_0}(\xi_0)\ge \min\{ c_{i,0}\} \sum \pi_{I_i}(\xi_0)=c_0 \pi_{I}(\xi_0)
\end{equation}
with $c_0>0$. By continuity of both $b_{\xi_0}$ and $\pi_I$ at $\xi_0$ and using \eqref{estimacionenxi0}
we can find a possibly smaller neighborhood $\TT_0$ satisfying the properties of Lemma \ref{vecindad} such that for $\xi\in \TT_0$ we have
\begin{equation}\label{continuidadcomutador}
\pi_I(\xi)b_{\xi_0}(\xi)\pi_I(\xi)\ge \tfrac{1}{2} c_0 \pi_{I}(\xi)\ .
\end{equation}

Since we chose $\xi_0$ arbitrarily in $\T^d$, we can construct $\TT_0$ satisfying \eqref{continuidadcomutador} for every
$\xi_0$. It follows that one can find a covering of the closed set $\UU_I:=p_{\T^d}(p_\R^{-1}(I))$ composed of a finite number of such $\TT_0$.
We have thus defined the covering $\{\TT_\ell\}$ already mentioned before the equation \eqref{ai} and mentioned in the above statement.
To finish, observe that $[ih,A_I]$ is a bounded fibered operator with fiber $b$ given for any $\xi\in \UU_I$ by
\begin{equation*}
b(\xi)=\sum_{\ell} \chi_\ell(\xi)b_{\xi_\ell}(\xi)\chi_\ell(\xi)\ .
\end{equation*}
Therefore, the operator $E_h(I)[ih,A_I]E_h(I)$ is a bounded fibered operator with fiber equal to
$ \pi_I(\xi)b(\xi)\pi_I(\xi)$.
We also infer that
\begin{equation*}
\sum_{\ell}\pi_I(\xi)\chi_\ell(\xi)b_{\xi_\ell}(\xi)\chi_\ell(\xi)\pi_I(\xi)\ge \tfrac{1}{2} \min_\ell \{{c_\ell}\} \pi_{I}(\xi)
\end{equation*}
for every $\xi \in \T^d$. By setting $a_I=\tfrac{1}{2} \min_\ell \{{c_\ell}\}$ we conclude that
\begin{equation*}
E_{h}(I)[ih,A_I]E_{h}(I)\ge a_I E_{h}(I).
\end{equation*}

Since the operator $B:=[ih,A_I]$ has been computed on $C^\infty(\T^d;\C^n)$ which is a core for $A_I$, and since the resulting
operator is bounded, one deduces from the results stated in Section \ref{subsec_Mourre} that $h$ belongs to $C^1(A_I)$.
Then, since the operator $B$ is again an analytically fibered operator, the computation of $[iB,A_I]$ can be performed
similarly on $C^\infty(\T^d;\C^n)$ and the resulting operator is once again bounded.
It then follows that $h$ belongs to $C^2(A_I)$.
\end{proof}

\begin{Remark}
When studying a particular graph one can usually find analytic families of eigenvalues $\lambda_i$ and
associated eigenprojections $\Pi_{i}$ outside a discrete subset of $\T^{d}$.
Then, a more natural conjugate operator is given formally by
$\sum\Pi_{i}\bigr((\partial\lambda_i)\cdot\nabla +\nabla \cdot (\partial\lambda_i)\bigl)\Pi_{i}$ as used for example in \cite{An13}
(see also \cite{GM} for a related construction).
In fact it is a classical result due to Rellich that for every one-dimensional analytic family of (not necessarily bounded) operators,
such analytic eigenprojections can be found. For dimension $2$, the theory of hyperbolic polynomials shows that this choice
can be made outside a discrete set \cite[Remark 5.6]{KP08}. For arbitrary dimension, there seems to be no argument to ensure that analytic
eigenprojections can be chosen and so we shall use the conjugate operator given by \eqref{ai}.
\end{Remark}

\section{Proof of the main theorem}\label{section:proofbis}
\setcounter{equation}{0}

In this section we provide the proof of our main theorem.
It will be divided into two subsections. In the first one we derive some abstract results which
are not directly linked with topological crystals. However, the form of the operators we consider is
inspired by the operators coming from the initial problem. In the second subsection
we show how the abstract results can be applied to the perturbation of the initial periodic operator on a crystal lattice.

Before starting with the first subsection let us recall that the general formula for a toroidal pseudodifferential operator on
$L^2(\T^d)$ has been introduced in \eqref{eq_pseudo}.
Subsequently, we shall need a slightly more general formula. Namely, the notion a toroidal pseudodifferential operator $\Op(a)$
acting on $u \in C^\infty(\T^d;\C^n)$ and given by
\begin{equation*}
[\Op(a)u](\xi) :=\sum_{\mu\in \Z^d} e^{-2\pi i\xi\cdot\mu} \;\! a(\xi,\mu)\;\! \check{u}(\mu), \qquad \xi \in \T^d,
\end{equation*}
where $a:\T^d\times \Z^d \to M_n(\C)$ is called its symbol.

\subsection{A few regular operators}

In the first lemma we derive the symbol corresponding to the adjoint of a special class of symbols.

\begin{Lemma}\label{lem_involutionbis}
For a bounded $a:\Z^d \to M_n(\C)$ and a fixed $\nu\in\Z^d$, we consider the symbol $a_{\nu}:\T^d\times \Z^d \to M_n(\C)$ defined by
\begin{equation}\label{eq_anu}
a_{\nu}(\xi,\mu)= e^{2\pi i\xi\cdot\nu}a(\mu), \qquad \forall \xi \in \T^d, \ \mu\in \Z^d,
\end{equation}
and the symbol $a_{\nu}^\dagger:\T^d\times \Z^d \to M_n(\C)$ defined by
\begin{equation}\label{eq_anu_dag}
a_{\nu}^\dagger(\xi,\mu)= e^{-2\pi i\xi\cdot\nu} a(\mu+\nu)^*, \qquad \forall \xi \in \T^d, \ \mu\in \Z^d.
\end{equation}
Then the following equality holds in $\B\big(L^2(\T^d;\C^n)\big)$:
$\Op(a_{\nu})^*=\Op(a_{\nu}^\dagger)$.
\end{Lemma}

\begin{proof}
For any $u,v\in C^\infty(\T^d;\C^n)$ one has
\begin{align*}
\la\Op(a_{\nu})u,v\ra=&\int_{\T^d}\d\xi\Big\la\sum_{\mu\in\Z^d} e^{-2\pi i \xi\cdot\mu}\int_{\T^d}\d\zeta  e^{2 \pi i \zeta\cdot\mu} a_{\nu}(\xi,\mu)u(\zeta),v(\xi)\Big\ra\\
=&\int_{\T^d}\d\xi\sum_{\mu\in\Z^d} e^{-2\pi i \xi\cdot(\mu-\nu)}\int_{\T^d}\d\zeta e^{2 \pi i \zeta\cdot\mu}\la a(\mu)u(\zeta),v(\xi)\ra\\
=&\int_{\T^d}\d\xi\sum_{\mu\in\Z^d} e^{-2\pi i \xi\cdot\mu}\int_{\T^d}\d\zeta e^{2 \pi i \zeta\cdot(\mu+\nu)}\la a(\mu+\nu)u(\zeta),v(\xi)\ra\\
=&\int_{\T^d}\d\zeta\sum_{\mu\in\Z^d} e^{2 \pi i \zeta\cdot(\mu+\nu)}\la a(\mu+\nu)u(\zeta),\check{v}(\mu)\ra\\
=&\int_{\T^d}\d\zeta\Big\la u(\zeta),\sum_{\mu\in\Z^d} e^{-2 \pi i \zeta\cdot\mu} e^{-2 \pi i \zeta\cdot\nu}a(\mu+\nu)^*\check{v}(\mu)\Big\ra\\
=&\la u,\Op(a_{\nu}^\dagger)v\ra\notag .\qedhere
\end{align*}
\end{proof}

Some additional operators will be necessary.
We denote by $N= (N_1,N_2,\dots,N_d)$ the position operators in $l^2(\Z^d;\C^n)$ acting
as $[N_j f](\mu)=\mu_j f(\mu)$ for any $f:\Z^d\to \C^n$ with compact support and for any $\mu\in \Z^d$.
For any $\nu\in \Z^d$ we also set $S_\nu$ for the shift operator by $\nu$ acting
on any $f\in l^2(\Z^d;\C^n)$ as $[S_\nu f](\mu)=f(\mu+\nu)$.
It is easily observed that the operators $N_j$
extend to self-adjoint operators in $l^2(\Z^d;\C^n)$ while $S_\nu$ is a unitary operator in this Hilbert space.
We start by treating the short range type of assumption on the symbol that ensures the regularity of the pseudodifferential operator.

\begin{Lemma}\label{lemmashortxbis}
Let $a:\Z^{d}\to M_n(\C)$ be such that
\begin{equation}\label{rscondicionxbis}
\int_{1}^{\infty}\d\lambda \sup_{\lambda<|\mu|<2\lambda}\lp a(\mu)\rp<\infty\ .
\end{equation}
Then for any fixed $\nu \in \Z^d$ the operator $\Op(a_{\nu}+a^\dagger_{\nu})$ belongs to $C^{1,1}(A_I)$, where $a_\nu$ and $a^\dagger_\nu$
have been defined respectively in \eqref{eq_anu} and in \eqref{eq_anu_dag}.
\end{Lemma}

Before the proof, let us mention that the interval on which the supremum is taken in \eqref{rscondicionxbis} is rather arbitrary.
Indeed, it is easily observed that this condition is equivalent to
\begin{equation*}
\int_{1}^{\infty}\d\lambda \sup_{c\lambda<|\mu|<c'\lambda}\lp a(\mu)\rp<\infty\
\end{equation*}
for any constants $c,c'$ satisfying $0<c<c'$. This flexibility will be useful several times in the following proofs.

\begin{proof}
This proof consists in an application of an abstract result for short-range type perturbations presented in \cite[Theorem 7.5.8]{ABG96}.
We shall thus check the assumptions of this theorem with $\mathscr{G} = \HH= L^2(\T^d;\C^n)$ and $\Lambda=(1-\Delta_{\T^d})^{\frac12}\otimes\Id_n$.
Condition $(1)$ corresponds to the boundedness of the unitary group generated by the self-adjoint operator $\Lambda$ in $\HH$.
Condition $(2)$ corresponds to the boundedness of the closure of the operator $\Lambda^{-2}A_I^2$ defined on the domain $\D(A_I^2)$.
Indeed, thanks to the material presented before Lemma \ref{lem_ess_adj} we know that $A_I^{2}$ is bounded from $\H^0(\T^d;\C^n)$ to
$\H^{-2}(\T^d;\C^n)$ while $\Lambda^{-2}$ is bounded from $\H^{-2}(\T^d;\C^n)$ to $\H^{}(\T^d;\C^n)$. Since $L^2(\T^d;\C^n)=\H^0(\T^d;\C^n)$
the mentioned condition is satisfied.
	
Since $\Op(a_{\nu}+a^\dagger_{\nu})$ is symmetric by Lemma \ref{lem_involutionbis} it only remains to show that there exists
$\theta\in C^{\infty}_c\big((0,\infty)\big)$ not identically zero such that
\begin{equation}\label{rsDobjetivoABGxbis}
\int_{1}^{\infty}\d\lambda\lp\theta\(\tfrac{\Lambda}{\lambda}\)\Op(a_{\nu}+a^\dagger_{\nu})\rp_{\B(\HH)} <\infty.
\end{equation}
For that purpose, let us first compute the operators $\F^*\Op(a_\nu)\F$ and $\F^*\Op(a_\nu^\dagger)\F$,
with $\F\equiv \F\otimes \Id_n$ the unitary Fourier transform from $l^2(\Z^d;\C^n)$ to $L^2(\T^d;\C^n)$.
A direct computation leads then to the equality
\begin{equation*}
\F^*\Op(a_\nu)\F=S_\nu a(N) \qquad \hbox{ and } \qquad \F^*\Op(a_\nu^\dagger)\F=a(N)^* S_{-\nu}
\end{equation*}
with $S_{\pm \nu}$ the translation operator introduced before the statement, and $a(N)$, resp.~$a(N)^*$,
the operator of multiplication by $a$, resp.~$a^*$, in $l^2(\Z^d;\C^n)$.
By using the unitarity of the Fourier transform one then obtains
for any function $\theta\in C^\infty(\R_+;[0,1])$ with support contained in $(\sqrt{2},2)$ that
\begin{align*}
&\lp\theta\(\tfrac{\Lambda}{\lambda}\)\Op(a_\nu+a_\nu^\dagger)\rp_{\B(\HH)} \\
& \leq \lp\theta\(\tfrac{\langle N\rangle}{\lambda}\)\F^*\Op(a_\nu)\F\rp_{\B(l^2(\Z^d;\C^n))}
+ \lp\theta\(\tfrac{\langle N\rangle}{\lambda}\)\F^*\Op(a_\nu^\dagger)\F\rp_{\B(l^2(\Z^d;\C^n))}\\
& = \lp\theta\(\tfrac{\langle N-\nu \rangle}{\lambda}\)a(N)\rp_{\B(l^2(\Z^d;\C^n))}
+ \lp\theta\(\tfrac{\langle N\rangle}{\lambda}\)a(N)^*\rp_{\B(l^2(\Z^d;\C^n))} \\
& \leq \sup_{2\lambda^2-1<|\mu-\nu|^2<4\lambda^2-1}\lp a(\mu)\rp
+ \sup_{2\lambda^2-1<|\mu|^2<4\lambda^2-1}\lp a(\mu)^*\rp \ .
\end{align*}
By using this final estimate and the comment made before the proof
one readily obtains that \eqref{rsDobjetivoABGxbis} is finite.
One has thus checked all the assumptions of \cite[Theorem 7.5.8]{ABG96}, from which one deduces that
$\Op(a_\nu+a_\nu^\dagger)$ belongs to $C^{1,1}(A_I)$.
\end{proof}

In the next lemma, we prove a slightly technical result which will be useful for the existence and the completeness of
the wave operators. For its statement the interpolation space $\NNN:=\big(\D(\langle N\rangle),l^2(\Z^d;\C^n)\big)_{\frac{1}{2},1}$
is necessary. Note that a precise description of this space is given in \cite[Thm.~3.6.2]{ABG96},
and that the following proof is inspired by a similar proof for Theorem 7.6.10~of the same reference.

\begin{Lemma}\label{lemma_wave}
Let $a:\Z^{d}\to M_n(\C)$ satisfy \eqref{rscondicionxbis}
and let $\nu\in \Z^d$ be fixed.
Then the operator $S_\nu a(N)$ belongs to $\B(\NNN^{*\circ},\NNN)$
where $\NNN^{*\circ}$ denotes the closure of $l^2(\Z^d;\C^n)$ in $\NNN^{*}$.
\end{Lemma}

\begin{proof}
We observe that for fixed $\nu\in \Z^d$ one can find $\theta, \tilde\theta\in C^\infty_c\big((0,\infty)\big)$ not identically zero
such that the equality $\theta\big(\frac{\langle \mu-\nu\rangle }{\lambda}\big) \tilde \theta\big(\frac{\langle \mu\rangle }{\lambda}\big)
=\theta\big(\frac{\langle \mu-\nu\rangle }{\lambda}\big)$ holds
for any $\mu\in \Z^d$ and $\lambda\geq 1$.
Then one infers that
\begin{align}
\label{term1} & \int_1^\infty \frac{\d \lambda}{\lambda} \lp \lambda^{1/2}\theta\(\tfrac{\langle N\rangle}{\lambda}\)S_\nu a(N)f\rp \\
\nonumber & = \int_1^\infty \frac{\d \lambda}{\lambda} \lp \lambda^{1/2}\theta\(\tfrac{\langle N-\nu\rangle}{\lambda}\) a(N)f\rp \\
\label{term2} &\leq  \Big(\int_1^\infty \d \lambda  \sup_{\mu}\lp \theta\(\tfrac{\langle \mu-\nu\rangle}{\lambda}\)a(\mu)\rp\Big)
\times \sup_{\lambda > 1}\lp \lambda^{-1/2} \tilde \theta \(\tfrac{\langle N\rangle}{\lambda}\)f\rp.
\end{align}
The term \eqref{term1} corresponds to the norm of $S_\nu a(N)f$ in the space $\NNN$
while the second factor in \eqref{term2} corresponds to the norm of $f$ in $\NNN^*$.
Since the first factor in \eqref{term2} is bounded (by the assumption and the remark made before the proof of Lemma \ref{lemmashortxbis}),
one deduces the statement.
\end{proof}

Before turning our attention to the long range type of assumption we state a simple result that can be thought
of as a discrete version of the fundamental theorem of calculus. In its statement, we use the norm $|\cdot|_1$
on $\Z^d$, namely $|\mu|_1=\sum_{j=1}^d|\mu_j|$. For any $\nu\in \Z^d$ and any $f:\Z^d\to \C$ we also set
\begin{equation*}
[\triangle_\nu f](\mu) = f(\mu+\nu) - f(\mu), \qquad \forall \mu \in \Z^d.
\end{equation*}

\begin{Lemma}
For any fixed $\nu\in\Z^d$ there exist $\{j_\ell\}_{\ell=1}^{|\nu|_1}\subset\{1,\dots,d\}$ and
$\{\gamma_\ell\}_{\ell=1}^{|\nu|_1}\subset\Z^d$ with $|\gamma_\ell|\leq|\nu|$ such that for any
$f:\Z^d\to \C$ one has
\begin{equation}\label{Laformule}
\triangle_\nu f=\sum_{\ell=1}^{|\nu|_1}\sgn(\nu_{j_\ell})(S_{\gamma_\ell}\triangle_{j_\ell} f)\ .
\end{equation}
\end{Lemma}

\begin{Lemma}\label{lemmalong}
Let $b:\Z^{d}\to\R$ be such that $\lim_{|\mu|\to \infty}b(\mu)=0$, and assume that for every $j\in\{1,\dots,d\}$
\begin{equation}\label{rlcondicion}
\int_{1}^{\infty}\d{\lambda}\sup_{\lambda<|\mu|<2\lambda}|(\triangle_j b)(\mu)|<\infty\ .
\end{equation}
Then, by setting $[b](\xi,\mu):=b(\mu)\Id_n$ for any $\xi\in \T^d$ and $\mu\in \Z^d$, the operator $\Op(b)$ is self-adjoint and belongs to $C^{1,1}(A_I)$.
\end{Lemma}

\begin{proof}
Since $b$ is real-valued, it directly follows from Lemma \ref{lem_involutionbis} that $\Op(b)$ is self-adjoint.

We shall now show that the commutator $[\Op(b),A_I]$, defined as a difference of operators on $C^\infty(\T^d;\C^n)$,
extends to an element of $\B\big(L^2(\T^d;\C^n)\big)$.
Since $A_I$ is a differential operator of order $1$, the exists a toroidal symbol $a\in S^1\big(\T^d\times \Z^d;M_n(\C)\big)$
such that $A_I=\Op(a)$, see Section \ref{subsec_A}.
It is then easily observed that the operator $\Op(a)\Op(b)$ coincides with $\Op(ab)$ with the symbol $ab$ given by
$a(\xi,\mu)b(\mu)$ for any $\xi\in \T^d$ and $\mu\in \Z^d$.
On the other hand, a few more computations show that the operator $\Op(b)\Op(a)$ is also a toroidal pseudodifferential operator $\Op(b\diamond a)$
with symbol
\begin{equation}\label{eq_sum}
b\diamond a(\xi,\mu) = \sum_{\nu\in \Z^d} e^{-2\pi i \xi\cdot \nu} b(\mu+\nu)\;\!\check a(\nu,\mu)
\end{equation}
where $\check a(\nu,\mu)=\int_{\T^d}\d \zeta e^{2\pi i \zeta\cdot \nu}a(\zeta,\mu)$.
Note that \eqref{eq_sum} is well-defined since the map $\zeta\mapsto a(\zeta,\mu)$ is smooth, and thus its inverse Fourier transform
is of Schwartz class on $\Z^d$.
We then observes that r.h.s~of \eqref{eq_sum} allows us to express the symbol $c$ of $[\Op(b),\Op(a)]$. Indeed one gets
\begin{align*}
c(\xi,\mu) :=b\diamond a(\xi,\mu) - b(\mu)a(\xi,\mu)
& = \sum_{\nu\in \Z^d} e^{-2\pi i \xi\cdot \nu} [\triangle_\nu b](\mu)\;\!\check a(\nu,\mu) \\
& = \F\big[\nu\mapsto [\triangle_\nu b](\mu)\;\!  \check a(\nu,\mu)\big](\xi)\ .
\end{align*}

A few computations show that
\begin{equation*}
[\F^* \Op(c)\F f](\mu) = \sum_{\nu\in \Z^d} \check c(\mu-\nu,\nu)f(\nu)
= \sum_{\nu\in \Z^d} [\triangle_{\mu-\nu} b](\nu)\;\!  \check a(\mu-\nu,\nu) f(\nu).
\end{equation*}
Thus, the operator $K :=\F^* \Op(c)\F$ is bounded if the map
$(\nu,\mu)\mapsto[\triangle_\nu b](\mu)\;\!  \check a(\nu,\mu) $
belongs to $l^1\big(\Z^d;l^\infty(\Z^d)\big)$, with $l^1$ for the $\nu$ variable and $l^\infty$ for the $\mu$ variable.
In order to show this property, recall that $a\in S^1\big(\T^d\times \Z^d;M_n(\C)\big)$ and by taking the equality \eqref{Laformule} into account,
observe that for any $q\in \N$
\begin{align}
\nonumber |[\triangle_\nu b](\mu)\;\!  \check a(\nu,\mu)|
& = \Big|\sum_{\ell=1}^{|\nu|_1}\sgn(\nu_{j_\ell})[S_{\gamma_\ell}\triangle_{j_\ell}b](\mu) \;\!  \check a(\nu,\mu)\Big| \\
\nonumber& \leq \sum_{\ell=1}^{|\nu|_1}\big|[S_{\gamma_\ell}\triangle_{j_\ell}b](\mu) \;\!  \langle \mu\rangle\langle \nu\rangle^{-q}\big|\
\big|\langle\nu\rangle^q \langle\mu\rangle^{-1}\check a(\nu,\mu)\big| \\
\nonumber & \leq C_q \sum_{\ell=1}^{|\nu|_1}\big|[S_{\gamma_\ell}\langle \cdot\rangle\triangle_{j_\ell}b](\mu)  \big|\
[\langle \mu+\gamma_\ell\rangle^{-1} \langle \mu\rangle]  \langle \nu\rangle^{-q} \\
\nonumber &\leq C_q \sum_{\ell=1}^{|\nu|_1}\Big\{\sup_{|\mu|-|\gamma_\ell|\leq|\beta|\leq|\mu|+|\gamma_\ell|}\langle \beta\rangle\big|[\triangle_{j_\ell}b](\beta)\big|\Big\}
\langle \gamma_\ell\rangle \langle \nu\rangle^{-q} \\
&\leq C_q \sum_{j=1}^{d}\Big\{\sup_{|\mu|-|\nu|\leq|\beta|\leq|\mu|+|\nu|}\langle \beta\rangle\big|[\triangle_{j}b](\beta)\big|\Big\}|\nu_j|
\langle \nu\rangle^{-q+1}\label{44} \\
\nonumber & \leq C_q  \langle \nu\rangle^{-q+2}
\end{align}
where $C_q$ is a constant which depends on $a$ and $b$ but not on $\mu$ or $\nu$, and which can be different from one line to another one.
Note that we have also used that $\sup_{\beta\in \Z^d}\langle \beta\rangle\big|[\triangle_{j}b](\beta)\big|<\infty$
for any $j\in \{1,\dots,d\}$, as a consequence of condition \eqref{rlcondicion}.
By choosing $q$ large enough, this expression belong to $l^1(\Z^d)$, which concludes the proof that $K$,
and consequently $[\Op(b),A_I]$, extend to bounded operators.

In order to apply \cite[Theorem 6.1]{BS99} we still need to show that
\begin{equation}\label{rlobjetivo}
\int_{1}^{\infty}\tfrac{\d\lambda}{\lambda}\lp\theta\(\tfrac{\Lambda}{\lambda}\)[i\Op(b),A_I]\rp_{\B(L^2(\T^d;\C^n))} <\infty
\end{equation}
for some function $\theta\in C^{\infty}_c\big((0,\infty)\big)$ not identically zero
and for the operator $\Lambda$ introduced in the proof of Lemma \ref{rscondicionxbis}.
In fact, by using the unitarity of the Fourier transform one then obtains
for any function $\theta\in C^\infty(\R_+;[0,1])$ that
\begin{align*}
\lp\theta\(\tfrac{\Lambda}{\lambda}\)[i\Op(b),A_I]\rp_{\B(L^2(\T^d;\C^n))}
& =  \lp\theta\(\tfrac{\Lambda}{\lambda}\)\Op(c)\rp_{\B(L^2(\T^d;\C^n))} \\
& = \lp\theta\(\tfrac{\langle N\rangle}{\lambda}\)K\rp_{\B(l^2(\Z^d;\C^n))} \\
& = \lp K \theta\(\tfrac{\langle N\rangle}{\lambda}\)\rp_{\B(l^2(\Z^d;\C^n))},
\end{align*}
where we have used in the last equality that $iK$ and $\theta\(\tfrac{\langle N\rangle}{\lambda}\)$ are self-adjoint.

Let us thus consider $\theta$ with support contained in $(r,s)$ with $r=\frac{11}{8}$ and $s=\frac{13}{8}$,
and let $K_\lambda$ denote the operator $K\theta \big(\frac{\langle N\rangle}{\lambda}\big)$. We shall again estimate the norm of $K_\lambda$ by estimating
the norm $l^1-l^\infty$ of the map
\begin{equation*}
(\nu,\mu)\mapsto \theta\(\tfrac{\langle \mu\rangle}{\lambda}\)[\triangle_\nu b](\mu)\;\!  \check a(\nu,\mu)\ .
\end{equation*}
By the previous computation we know that $\lambda\to\lp K_\lambda\rp_{1,\infty}$ is bounded so we need only to study its behavior when $\lambda\to\infty$.
For that purpose, observe first that
\begin{align*}
\lp K_\lambda\rp_{1,\infty}&=\sum_{\nu\in\Z^d}\sup_{\mu\in\Z^d}|\theta\(\tfrac{\langle \mu\rangle}{\lambda}\)[\triangle_\nu b](\mu)\;\!  \check a(\nu,\mu)|\\
&=\sum_{\nu\in\Z^d}\sup_{\sqrt{r^2\lambda^2-1}<|\mu|<\sqrt{s^2\lambda^2-1}}|[\triangle_\nu b](\mu)\;\!  \check a(\nu,\mu)|,
\end{align*}
and let us divide this sum in two parts. For the first we use the estimate \eqref{44}:
\begin{align*}
&\sum_{|\nu|<\frac{\lambda}{4}}\sup_{\sqrt{r^2\lambda^2-1}<|\mu|<\sqrt{s^2\lambda^2-1}}|[\triangle_\nu b](\mu)\;\!  \check a(\nu,\mu)| \\
&\leq \sum_{|\nu|<\frac{\lambda}{4}}\sup_{\sqrt{r^2\lambda^2-1}<|\mu|<\sqrt{s^2\lambda^2-1}}\( C_q \sum_{j=1}^{d}\Big\{\sup_{|\mu|-|\nu|\leq|\beta|\leq|\mu|+|\nu|}\langle \beta\rangle\big|[\triangle_{j}b](\beta)\big|\Big\}|\nu_j|
\langle \nu\rangle^{-q+1}\)\\
&\leq C_q\sum_{|\nu|<\frac{\lambda}{4}}\sup_{\sqrt{r^2\lambda^2-1}<|\mu|<\sqrt{s^2\lambda^2-1}}\(\langle \mu\rangle  \sum_{j=1}^{d}\Big\{\sup_{|\mu|-|\nu|\leq|\beta|\leq|\mu|+|\nu|}\big|[\triangle_{j}b](\beta)\big|\Big\}|\nu_j|
\langle \nu\rangle^{-q+2}\)\\
&\leq C_q\lambda\sum_{|\nu|<\frac{\lambda}{4}}\sup_{\sqrt{r^2\lambda^2-1}<|\mu|<\sqrt{s^2\lambda^2-1}}\(  \sum_{j=1}^{d}\Big\{\sup_{|\mu|-|\nu|\leq|\beta|\leq|\mu|+|\nu|}\big|[\triangle_{j}b](\beta)\big|\Big\}|
\langle \nu\rangle^{-q+3}\)\\
&\leq C_q\lambda\sum_{\nu<\frac{\lambda}{4}}\sum_{j=1}^{d}\sup_{\sqrt{r^2\lambda^2-1}-|\nu|<|\beta|<\sqrt{s^2\lambda^2-1}+|\nu|}\(  \big|[\triangle_{j}b](\beta)\big|\)|
\langle \nu\rangle^{-q+3}\\
&\leq C_q\lambda\sum_{j=1}^{d}\sup_{\sqrt{r^2\lambda^2-1}-\frac{\lambda}{4}<|\beta|<\sqrt{s^2\lambda^2-1}+\frac{\lambda}{4}} \big|[\triangle_{j}b](\beta)\big|,
\end{align*}
where, for $q=d+4$, the estimate
$$
\sum_{|\nu|<\frac{\lambda}{4}}\langle \nu\rangle^{-q+3}\leq \sum_{\nu\in \Z^d}\langle \nu\rangle^{-q+3}<\infty
$$
has been used. Note that in the previous computation, $C_q$ is a constant which does not depend on $\mu$ or $\nu$,
but which can be different from one line to another one.

For the other part of the summation, we just compute
\begin{align*}
&\sum_{|\nu|\geq\frac{\lambda}{4}}\sup_{\sqrt{r^2\lambda^2-1}<|\mu|<\sqrt{s^2\lambda^2-1}}|[\triangle_\nu b](\mu)\;\!  \check a(\nu,\mu)|\\
&\leq \sum_{|\nu|\geq\frac{\lambda}{4}}\sup_{\sqrt{r^2\lambda^2-1}<|\mu|<\sqrt{s^2\lambda^2-1}}\sum_{\ell=1}^{|\nu|_1}|[S_{\gamma_\ell}\triangle_{j_\ell} b](\mu)\;\!  \check a(\nu,\mu)|\\
&\leq \max_j\lp\triangle_jb\rp_\infty\sum_{|\nu|\geq\frac{\lambda}{4}}\sup_{\sqrt{r^2\lambda^2-1}<|\mu|<\sqrt{s^2\lambda^2-1}}\langle \mu\rangle\langle \nu\rangle^{-q+1} \; \big|\langle\nu\rangle^{q}\langle\mu\rangle^{-1}\check a(\nu,\mu)\big|\\
&\leq C_q'\lambda\sum_{|\nu|\geq\frac{\lambda}{4}}\la\nu\ra^{-q+1}
\end{align*}
with $C_q'$ a constant which does not depend on $\mu$ or $\nu$.
Hence for $\lambda$ large enough and still for $q=d+4$ we get
\begin{align*}
\lp K_\lambda\rp_{1,\infty}&\leq C_q\lambda\sum_{j=1}^{d}\sup_{\sqrt{r^2\lambda^2-1}-\frac{\lambda}{4}<|\beta|<\sqrt{s^2\lambda^2-1}+\frac{\lambda}{4}}
\big|[\triangle_{j}b](\beta)\big|+ C_q'\lambda\sum_{|\nu|\geq\frac{\lambda}{4}}\la\nu\ra^{-d-3}\\
&\leq C_q \lambda\sum_{j=1}^{d}\sup_{\lambda<|\beta|<2\lambda}\big|[\triangle_{j}b](\beta)\big|+ C_q'\lambda^{-2}.
\end{align*}
By finally taking into account the inequality $\lp\theta\(\frac{\Lambda}{\lambda}\)[i\Op(b),A_I]\rp_{\B(L^2(\T^d;\C^n))} \leq\lp K_\lambda\rp_{1,\infty}$
and the assumption \eqref{rlcondicion} one concludes that \eqref{rlobjetivo} is finite.
By applying the statement of \cite[Theorem 6.1]{BS99}, one deduces that $\Op(b)$ belongs to $C^{1,1}(A_I)$.
\end{proof}

\subsection{Regularity of the perturbations and proof of Theorem \ref{principalvertices}}

Since the operator $H_0$ is unitarily equivalent to a bounded analytically fibered operator in the Hilbert space $L^2(\T^d;\C^n)$,
the second step consists in performing a similar transformation to the operator $\J H \J^*$, where $\J$ was introduced in \eqref{eq_def_J}.
For that purpose, recall first that the maps $\U$ and $\I$ have been introduced respectively in \eqref{def_de_U} and in \eqref{def_de_I}.

\begin{Proposition}
The difference $\I\U\big(\Delta(X,m_0)-\J\Delta(X,m)\J^*\big)\U^*\I^*$ is a toroidal pseudodifferential operator.
Moreover, its symbol $b:\T^d\times\Z^d\to M_{n}(\C)$ is given by
\begin{equation}\label{def_de_b}
b(\xi,\mu): =\sum_{\ee\in\A(\XX)}\Big([T(\ee)](\mu) - e^{2\pi i \xi \cdot \eta(\ee)} [K(\ee)](\mu)\Big)
\end{equation}
with $K(\ee):\Z^d\to M_n(\C)$ and $T(\ee):\Z^d\to M_n(\C)$ defined by
\begin{equation}\label{K1}
\[K(\ee)\](\mu)_{j\ell}:=\left\{\begin{matrix} \(\frac{m((\mu-\eta(\ee))\hat{\ee})}{m((\mu-\eta(\ee)) o(\hat{\ee}))^{\frac12}m((\mu-\eta(\ee)) t(\hat{\ee}))^{\frac12}}-\frac{m_0(\ee)}{m_0(o(\ee))^{\frac12}m_0(t(\ee))^{\frac12}}\) & \hbox{ if } \ee = (\xx_j,\xx_\ell) \\
0 &  \hbox{otherwise}\end{matrix}\right.
\end{equation}
and
\begin{equation}\label{K2}
\[T(\ee)\](\mu)_{j \ell}:=\left\{\begin{matrix} \(\frac{m(\mu\hat{\ee})}{m(\mu o(\hat{\ee}))}-\frac{m_0(\ee)}{m_0(o(\ee))}\) & \hbox{ if } o(\ee)=\xx_j \hbox{ and } j=\ell \\
0 &  \hbox{otherwise} \end{matrix} \right.
\end{equation}
\end{Proposition}

Before the proof, let us introduce the following convenient map:
\begin{equation*}
\imath:V(X)\to\{1,\dots,n\}, \qquad x_{\imath(x)}:=\widehat{\check{x}},
\end{equation*}
which associates to any $x\in V(X)$ the index of the representative $x_j\in V(X)$
which belongs to the same orbit under the action of $\Z^d$.

\begin{proof}
By a direct computation one first obtains an explicit expression for the operator $\J \Delta(X,m) \J^*$, namely for any $f\in l^2(X,m_0)$ and $x\in V(X)$,
\begin{equation*}
[\J \Delta(X,m) \J^*f](x)=\sum_{\e\in\A_x}\frac{m(\e)}{m(x)^{\frac{1}{2}}m(t(\e))^{\frac{1}{2}}} \frac{m_0(t(\e))^{\frac{1}{2}}}{m_0(x)^{\frac{1}{2}}}f\big(t(\e)\big)
-\deg_m(x)f(x) \ .
\end{equation*}
In particular, by using the unique decomposition introduced in Section \ref{sec_H_0}
as well as the equalities \eqref{eq_use_ind} one has for $x=\mu x_j$
\begin{align*}
[\J \Delta(X,m) \J^*f](\mu x_j)& =\sum_{\e \in \A_{\mu x_j}}\frac{m(\e)}{m(\mu x_j)^{\frac{1}{2}}m(t(\e))^{\frac{1}{2}}} \frac{m_0(t(\e))^{\frac{1}{2}}}{m_0(\mu x_j)^{\frac{1}{2}}}f\big(t(\e)\big)-\deg_m(\mu x_j)f(\mu x_j) \\
& =\sum_{\e \in \A_{x_j}}\!\!\!\frac{m(\mu\e)}{m(\mu o(\e))^{\frac{1}{2}}m(\mu t(\e))^{\frac{1}{2}}} \frac{m_0(\mu t(\e))^{\frac{1}{2}}}{m_0(\mu o(\e))^{\frac{1}{2}}}f\big(\mu t(\e)\big)-\deg_m(\mu x_j)f(\mu x_j) \\
& =\sum_{\ee \in \A_{\xx_j}}\frac{m(\mu\hat{\ee})}{m(\mu o(\hat{\ee}))^{\frac{1}{2}}m(\mu t(\hat{\ee}))^{\frac{1}{2}}} \frac{m_0(t(\ee))^{\frac{1}{2}}}{m_0(o(\ee))^{\frac{1}{2}}}f\big(\mu t(\hat{\ee})\big)-\deg_m(\mu x_j)f(\mu x_j) \ .
\end{align*}

Now,  by taking into account the explicit form of $\U$ and $\I$, and by identifying $u\in C^\infty(\T^d;\C^n)$ with
$(u_1,\dots,u_n)$ with each $u_j\in C^\infty(\T^d)$ one infers that
\begin{align*}
& [\I \U \J \Delta(X,m)\J^* \U^* \I^* u]_j(\xi) \\
& = \sum_{\mu\in \Z^d} e^{-2\pi i \xi\cdot \mu}\sum_{\ee\in\A_{\xx_j}} \frac{m(\mu \hat \ee)}{m(\mu o(\hat{\ee}))^{\frac{1}{2}}m(\mu t(\hat{\ee}))^{\frac{1}{2}}}
\check{u}_{\imath(t(\ee))}\big(\mu+\eta(\ee)\big)\notag \\
& \quad - \sum_{\mu\in \Z^d} e^{-2\pi i \xi\cdot \mu} \deg_m(\mu x_j)\check{u}_j(\mu) \\
& = \sum_{\mu\in \Z^d} e^{-2\pi i \xi\cdot \mu}\sum_{\ee\in\A_{\xx_j}} \frac{m((\mu-\eta(\ee)) \hat \ee)}{m((\mu-\eta(\ee))o(\hat{\ee}))^{\frac{1}{2}}m((\mu-\eta(\ee)) t(\hat{\ee}))^{\frac{1}{2}}}
e^{2\pi i \xi\cdot \eta(\ee)}\check{u}_{\imath(t(\ee))}(\mu)\notag \\
& \quad - \sum_{\mu\in \Z^d} e^{-2\pi i \xi\cdot \mu} \Big(\sum_{\ee\in \A_{\xx_j}}\frac{m(\mu \hat \ee)}{m(\mu o(\hat \ee))}\Big)\check{u}_j(\mu)\ ,
\end{align*}
where the definition of the degree provided in \eqref{eq_def_deg} has used for the last equality.
Clearly, this operator corresponds to a toroidal pseudodifferential operator. It then only remains to combine this expression with
\eqref{def_h} and one deduces that the operator $\I\U\big(\Delta(X,m_0)-\J\Delta(X,m)\J^*\big)\U^*\I^*$ is also a toroidal pseudodifferential operator
whose symbol is given by \eqref{def_de_b}.
\end{proof}

The precise formula for the symbol $b$ is useful because one can now apply Lemma \ref{lemmashortxbis} to see that $\I\U\big(\Delta(X,m_0)-\J\Delta(X,m)\J^*\big)\U^*\I^*$
belongs to $C^{1,1}(A_I)$.

\begin{Lemma}
Assume that the measure $m$ satisfies the condition \eqref{measureshort}. Then the difference $\I\U\big(\Delta(X,m_0)-\J\Delta(X,m)\J^*\big)\U^*\I^*$ belongs to $C^{1,1}(A_I)$.
\end{Lemma}

\begin{proof}
Observe first that the symbol $b$ of the previous statement satisfies
\begin{align*}
b(\xi,\mu) & =\sum_{\ee\in\A(\XX)}\Big([T(\ee)](\mu) - e^{2\pi i \xi \cdot \eta(\ee)} [K(\ee)](\mu)\Big) \\
& = \sum_{\ee\in\A(\XX)}[T(\ee)](\mu) - \frac{1}{2}\sum_{\ee\in\A(\XX)}\Big( e^{2\pi i \xi \cdot \eta(\ee)} [K(\ee)](\mu)
+ e^{2\pi i \xi \cdot \eta(\bar\ee)} [K(\bar\ee)](\mu)\Big).
\end{align*}
By keeping in mind that
\begin{equation}\label{eq_to_remember}
\eta(\ee)\hat{\overline{\ee}}=\overline{\hat{\ee}},\quad \eta(\ee)o(\hat{\overline{\ee}})=t(\hat{\ee}), \quad \hbox{and}\quad
\eta(\ee)t(\hat{\overline{\ee}})=o(\hat{\ee})
\end{equation}
one observes that for $\ee=(\xx_j,\xx_\ell)$
\begin{align*}
[K(\bar\ee)](\mu)_{\ell j} & =\frac{m((\mu-\eta(\bar\ee))\hat{\bar\ee})}{m((\mu-\eta(\bar\ee)) o(\hat{\bar\ee}))^{\frac12}m((\mu-\eta(\bar\ee)) t(\hat{\bar\ee}))^{\frac12}}-\frac{m_0(\bar\ee)}{m_0(o(\bar\ee))^{\frac12}m_0(t(\bar\ee))^{\frac12}} \\
& = \frac{m((\mu+\eta(\ee))\hat{\bar\ee})}{m((\mu+\eta(\ee)) o(\hat{\bar\ee}))^{\frac12}m((\mu+\eta(\ee)) t(\hat{\bar\ee}))^{\frac12}}-\frac{m_0(\ee)}{m_0(t(\ee))^{\frac12}m_0(o(\ee))^{\frac12}} \\
& = \frac{m(\mu\overline{\hat{\ee}})}{m(\mu t(\hat{\ee}))^{\frac12}m(\mu o(\hat{\ee}))^{\frac12}}-\frac{m_0(\ee)}{m_0(t(\ee))^{\frac12}m_0(o(\ee))^{\frac12}} \\
& = \frac{m(\mu \hat{\ee})}{m(\mu o(\hat{\ee}))^{\frac12}m(\mu t(\hat{\ee}))^{\frac12}}-\frac{m_0(\ee)}{m_0(o(\ee))^{\frac12}m_0(t(\ee))^{\frac12}} \\
& = [K(\ee)](\mu+\eta(\ee))_{j\ell}\ .
\end{align*}
By using the notation of Lemma \ref{lem_involutionbis} one also deduces that
\begin{align*}
e^{2\pi i \xi \cdot \eta(\ee)} [K(\ee)](\mu) + e^{2\pi i \xi \cdot \eta(\bar\ee)} [K(\bar\ee)](\mu)
& =  e^{2\pi i \xi \cdot \eta(\ee)} [K(\ee)](\mu) + e^{-2\pi i \xi \cdot \eta(\ee)} \big([K(\ee)](\mu+\eta(\ee))\big)^* \\
& = \big[K(\ee)_{\eta(\ee)}\big](\xi,\mu) + \big[K(\ee)_{\eta(\ee)}^\dagger\big](\xi,\mu) \\
& = \Big(K(\ee)_{\eta(\ee)} + K(\ee)_{\eta(\ee)}^\dagger\Big)(\xi,\mu) \ ,
\end{align*}
and by summing up these information, one has thus obtained that
\begin{equation}\label{petitxchouxfleursauchocolat}
b= \sum_{\ee\in\A(\XX)} \Big(T(\ee) - \frac{1}{2}\big(K(\ee)_{\eta(\ee)} + K(\ee)_{\eta(\ee)}^\dagger\big) \Big)\ .
\end{equation}

We are thus in a suitable position for using Lemma \ref{lemmashortxbis}, and it remains to show that
the condition \eqref{measureshort} implies the condition \eqref{rscondicionxbis} for the corresponding function $a$.
Since the sum in \eqref{petitxchouxfleursauchocolat} is finite, we can consider the contribution due to each $\ee$ separately.
Let us fix $\ee\in \A(\XX)$ and set for any $\mu \in\Z^d$: $f(\mu):=\frac{m((\mu-\eta(\ee))\hat{\ee})}{m((\mu-\eta(\ee)) o(\hat{\ee}))}$,
$g(\mu):=\frac{m(\mu\hat{\overline{\ee}})}{m(\mu o(\hat{\overline{\ee}}))}$, $f_0(\mu):=\frac{m_0(\ee)}{m_0(o(\ee))}$ and
$g_0(\mu):=\frac{m_0(\overline{\ee})}{m_0(o(\overline{\ee}))}$, the last two expressions being clearly independent of $\mu$.
Then, by taking the relations \eqref{eq_to_remember} into account one deduces that
\begin{align*}
\lp \[K(\ee)\](\mu)\rp
&=\left|f(\mu)^\frac12 g(\mu)^\frac12 -f_0(\mu)^\frac12 g_0(\mu)^\frac12\right|\\
&=\left|\big(f(\mu)-f_0(\mu)\big)\frac{g(\mu)^\frac12}{f(\mu)^\frac12+f_0(\mu)^\frac12}
+\big(g(\mu)-g_0(\mu)\big)\frac{f_0(\mu)^\frac12}{g(\mu)^\frac12+g_0(\mu)^\frac12}\right|\ .
\end{align*}
Since the functions
$\frac{g^\frac12}{f^\frac12+f_0^\frac12}$ and $\frac{f_0^\frac12}{g^\frac12+g_0^\frac12}$ are bounded on $\Z^d$  we finally obtain
\begin{align*}
\sup_{\lambda<|\mu|<2\lambda }\lp \[K(\ee)\](\mu)\rp
& \leq C \Big( \sup_{\lambda<|\mu|<2\lambda }|f(\mu)-f_0(\mu)|+\sup_{\lambda<|\mu|<2\lambda }|g(\mu)-g_0(\mu)|\Big) \\
& \leq C \Big( \sup_{\lambda<|\mu|<2\lambda }\Big|\frac{m((\mu-\eta(\ee))\hat{\ee})}{m((\mu-\eta(\ee)) o(\hat{\ee}))}-\frac{m_0(\ee)}{m_0(o(\ee))}\Big| \\
& \qquad \qquad + \sup_{\lambda<|\mu|<2\lambda }\Big|\frac{m(\mu\hat{\overline{\ee}})}{m(\mu o(\hat{\overline{\ee}}))}
-\frac{m_0(\overline{\ee})}{m_0(o(\overline{\ee}))}\Big|\Big).
\end{align*}
By taking into account the invariance of condition \eqref{measureshort} under a finite shift, one deduces from this condition and from the
previous computation that
\begin{equation}\label{eq_cond_K}
\int_1^\infty \d \lambda  \sup_{\lambda<|\mu|<2\lambda }\lp \[K(\ee)\](\mu)\rp <\infty.
\end{equation}
As a consequence of Lemma \ref{lemmashortxbis} it means that $\Op\big(K(\ee)_{\eta(\ee)} + K(\ee)_{\eta(\ee)}^\dagger\big) \in C^{1,1}(A_I)$.

For $T(\ee)$ the situation is much simpler. Clearly, $[T(\ee)](\mu)$ is a self-adjoint matrix for any $\mu\in \Z^d$.
By using again the notation introduced in Lemma \ref{lem_involutionbis} one infers that
$T(\ee)\equiv T(\ee)_{\mathbf 0} = T(\ee)_{\mathbf 0}^\dagger $ which implies that $\Op\big(T(\ee)\big)$ is self-adjoint.
In addition, it easily follows from  the assumption \eqref{rscondicionxbis} that
\begin{equation}\label{eq_dec_T}
\int_1^\infty \d \lambda  \sup_{\lambda<|\mu|<2\lambda }\lp \[T(\ee)\](\mu)\rp <\infty,
\end{equation}
which corresponds to the condition \eqref{rscondicionxbis} of Lemma \ref{lemmashortxbis}.
It thus follows that $\Op(b)$ belongs to $C^{1,1}(A_I)$, which corresponds to the statement of the lemma.
\end{proof}

We now turn our attention to the multiplicative perturbation. Since $\J R\J^*=R$ we can directly consider the operator $R-R_0$ in $l^2(X,m_0)$.

\begin{Lemma}
Assume that the difference $R-R_0$ is equal to $R_s+R_l$ and that these functions satisfy \eqref{Rshipotesis} and \eqref{Rlhipotesis}.
Then $\I\U(R-R_0)\U^*\I^*$ belongs to $C^{1,1}(A_I)$.
\end{Lemma}

\begin{proof}
Let us first set for any $x\in V(X)$
\begin{equation*}
\tilde{R}_s(x):=R_s(x)+\big(R_l(x)-R_l([x]x_1)\big)\quad \hbox{and}\quad
\tilde{R}_l(x):=R_l([x]x_1)
\end{equation*}
which implies that $R=\tilde{R}_s+\tilde{R}_l$. Note that a similar decomposition in the continuous case was already used in \cite{GN98a}.
The terms $\tilde{R}_s$ and $\tilde{R}_l$ will be treated separately, starting with $\tilde{R}_s$.

By some easy computations we get that $\I\U\tilde{R}_s\U^*\I^* = \Op(r_s)$, with the symbol $r_s:\Z^d\to M_n(\C)$ (and thus independent of the variable $\xi$) given by
\begin{equation}\label{rbbscalculado}
r_s(\mu)_{j\ell}=\big(R_s(\mu x_j)+R_l(\mu x_j)-R_l(\mu x_1)\big)\delta_{j\ell}.
\end{equation}
Since $r_s(\mu)=r_s(\mu)^*$ for any $\mu\in \Z^d$, we only need to show that $r_s$ satisfies \eqref{rscondicionxbis} in order to apply the content of Lemma \ref{lemmashortxbis}.
In fact, a sufficient condition is to show that for any $j\in \{1,\dots,n\}$ one has
\begin{equation}\label{ukulelebakalele}
\int_1^\infty \d \lambda  \sup_{\lambda<|\mu|<2\lambda }| r_s(\mu)_{jj}| <\infty.
\end{equation}

For that purpose, let us consider for any $j\in \{2,3,\dots n\}$ a fixed finite path $\alpha_j=\{\e_{j,p}\}_{p=1}^{N_j}$ between $x_1$ and $x_j$.
By applying $\mu$ to each edge in $\alpha_j$ we get a path between $\mu x_1$ and $\mu x_j$.
We will then use the fact that for a path $\alpha=\{\e_p\}_{p=1}^{N}$ between two vertices $x$ and $y$,
\ie~$o(\e_{1})=x$, $t(\e_{p})=o(\e_{p+1})$ and $t(\e_{N})=y$, the following formula holds for every $f\in C^{0}(X)$:
\begin{equation*}
f(y)-f(x)=\sum_{\e\in\alpha}\Big(f\big(t(\e)\big)-f\big(o(\e)\big)\Big)\ .
\end{equation*}
Keeping this notation in mind we can compute
\begin{align*}
|r_s(\mu)_{jj}| & = |R_s(\mu x_j)+R_l(\mu x_j)-R_l(\mu x_1)|\\
& \leq |R_s(\mu x_j)|+\sum_{p=1}^{N_j}\big|R_l\big(t(\mu\e_{j,p})\big)-R_l\big(o(\mu\e_{j,p})\big)\big|.
\end{align*}

Clearly, as a consequence of assumption \eqref{Rshipotesis} the first term satisfies
\begin{equation*}
\int_1^\infty \d \lambda  \sup_{\lambda<|\mu|<2\lambda }| R_s(\mu x_j)| \leq \int_1^\infty \d \lambda  \sup_{\lambda<|[x]|<2\lambda }| R_s(x)|<\infty.
\end{equation*}
On the other hand, as a consequence of assumption \eqref{Rlhipotesis} and its invariance under translations one also infers that
\begin{equation}\label{Ahjenesaisplusquoiinventer}
\int_1^\infty \d \lambda  \sup_{\lambda<|\mu|<2\lambda } \big|R_l\big(t(\mu\e_{j,p})\big)-R_l\big(o(\mu\e_{j,p})\big)\big|<\infty
\end{equation}
for any $\e_{j,p}$. One then deduces that the estimate \eqref{ukulelebakalele} holds,
and by applying Lemma \ref{lemmashortxbis} one gets that the operator $\Op(r_s)$ belongs to $C^{1,1}(A_I)$.

For the term $\tilde{R}_l$ we first observe with the notations of Lemma \ref{lemmalong} that
$\I\U\tilde{R}_l\U^*\I^* = \Op(r_l\Id_n)$, with the symbol $r_l:\Z^d\to \R$ defined by
\begin{equation}\label{rbblcalculado}
r_l(\mu)=R_l(\mu x_1)\ .
\end{equation}
It remains to show that $r_l$ satisfies the conditions of Lemma \ref{lemmalong}.
For that purpose let $\{\delta_j\}_{j=1}^d$ denote the canonical base of $\Z^{d}$.
We fix $\beta_j=\{\e_{j,p}\}_p^{N_j}$ a path between $x_1$ and $\delta_jx_1$.
By applying $\mu$ to each edge in $\beta_j$ we get a path between $\mu x_1$ and $(\mu+\delta_j)x_1$. Therefore we have
\begin{equation*}
|[\triangle_j r_l](\mu)| =|R_l\big((\mu+\delta_j)x_1\big)-R_l(\mu x_1)| \leq\sum_{p=1}^{N_j}\big|R_l\big(t(\mu\e_{j,p})\big)-R_l\big(o(\mu\e_{j,p})\big)\big|.
\end{equation*}
By invoking the same argument as before it follows from \eqref{Ahjenesaisplusquoiinventer} that the assumption \eqref{rlcondicion} of Lemma \ref{lemmalong}
is satisfied. By applying this lemma, it follows that $\Op(r_l\Id_n)$ belongs to $C^{1,1}(A_I)$, as expected.
\end{proof}

\begin{proof}[Proof of Theorem \ref{principalvertices}]
As a consequence of the previous lemmas the difference
\begin{equation}\label{eq_difference}
\I\U\big(\J H\J^*-H_0\big)\U^*\I^*
\end{equation}
belongs to $C^{1,1}(A_I)$. Moreover, it follows from the arguments presented in the proofs of these lemmas that
the difference \eqref{eq_difference} can be written as a finite sum of simpler operators, each of them being compact and self-adjoint.
One thus infers that the operator \eqref{eq_difference} is a compact operator and self-adjoint in $L^2(\T^d;\C^n)$.
We are thus in a suitable position for using Theorem \ref{mourreap}
with $S=h_0$ and $V$ defined by \eqref{eq_difference}. For $\tilde{\mu}^{A_I}(h)$, one can use the result obtained in \eqref{eq_mu},
by considering a slightly bigger interval $I'$ with $I\subset I'\subset \R\setminus \tau$.
Then, the points 1. and 2. of Theorem \ref{principalvertices} follow from Theorem \ref{mourreap} by taking into account the conjugation by the unitary
transform $\I\U$.

For the existence and asymptotic completeness of the wave operators, observe first that since $\J$ is unitary, these properties for $W_\pm(H,H_0;\J^*,I)$ are equivalent to
the same properties for $W_\pm(\J H\J^*,H_0;I)$. Then, by using again the unitary transform $\I\U$, one observes that this is still equivalent to
the existence and the asymptotic completeness of
\begin{equation}\label{eq_local}
W_\pm(\I\U \J H\J^* \U^*\I^*,\I\U H_0 \U^*\I^*;I).
\end{equation}
Such properties will now be deduced from \cite[Theorem 7.4.3]{ABG96}. Indeed, according to
that statement, if the difference \eqref{eq_difference} belongs to $\B(\KK^{*\circ},\KK)$,
with $\KK=\big(\D(A_I),L^2(\T^d;\C^n)\big)_{\frac{1}{2},1}$ and $\KK^{*\circ}$ the closure of $L^2(\T^d;\C^n)$ in $\KK^*$,
then the local wave operators \eqref{eq_local} exist and are asymptotically complete.

In order to check this condition, recall that the operator $\Lambda:=\big(\Id-\Delta_{\T^d}\big)\otimes \Id_n$
had been introduced in the proof of Lemma \ref{lem_ess_adj}, and as a consequence of Nelson's commutator theorem
one has $\D(\Lambda)\subset \D(A_I)$.
It then follows that $\LL:=\big(\D(\Lambda),L^2(\T^d;\C^n)\big)_{\frac{1}{2},1}\subset \big(\D(A_I),L^2(\T^d;\C^n)\big)_{\frac{1}{2},1}$, \
as shown for example in \cite[Corol.~2.6.3]{ABG96},
and then $\B(\LL^{*\circ},\LL)\subset \B(\KK^{*\circ},\KK)$.
However, we shall still consider the Fourier transform version of the spaces.
More precisely, let us set $\NNN:=\F^* \big(\D(\Lambda),L^2(\T^d;\C^n)\big)_{\frac{1}{2},1}$
which is equal to $\big(\D(\langle N\rangle),l^2(\Z^d;\C^n)\big)_{\frac{1}{2},1}$. Accordingly, one has to show that
\begin{equation}\label{eq_inclusion}
\F^*\I\U\big(\J H\J^*-H_0\big)\U^*\I^*\F \ \in \B(\NNN^{*\circ},\NNN).
\end{equation}

Fortunately, the l.h.s.~has already been computed and corresponds to
\begin{equation}\label{finite_sum}
\sum_{\ee\in \A(\XX)}\Big([T(\ee)](N)- S_{\eta(\ee)}[K(\ee)](N)\Big) + r_s(N)
\end{equation}
with $K(\ee)$ and $T(\ee)$ introduced respectively in \eqref{K1} and \eqref{K2}, and $r_s$ introduced in \eqref{rbbscalculado} when $R_l=0$.
We also recall that $S_{\eta(\ee)}$ denotes the shift operator by $\eta(\ee)$.
In addition, each of these terms satisfy an estimate of the form
\begin{equation*}
\int_1^\infty \d \lambda  \sup_{\lambda<|\mu|<2\lambda }\lp V(\mu)\rp <\infty,
\end{equation*}
with $V$ replacing $K(\ee)$, $T(\ee)$ or $r_s$, as shown in \eqref{eq_cond_K}, \eqref{eq_dec_T}, and \eqref{ukulelebakalele}.
Thus, we can apply Lemma \ref{lemma_wave} and deduce that all operator
$S_{\eta(\ee)}[K(\ee)](N)$, $[T(\ee)](N)$ and $r_s(N)$ belong to $\B(\NNN^{*\circ},\NNN)$.
Since the summation in \eqref{finite_sum} is finite, one concludes that the inclusion in \eqref{eq_inclusion} indeed holds.
\end{proof}

\end{document}